\documentclass[11pt]{article}
\usepackage{amssymb, amsmath, tikz-cd, amsthm,amsfonts}
\usepackage{mathtools}

\newcommand{\nn}{\mathbb N}
 
\newcommand{\rr}{\mathbb R}

\newcommand{\norm}[1]{\left\lVert#1\right\rVert}

\newtheorem{thm}{Theorem}
\newtheorem{prop}{Proposition}
\newtheorem{lem}{Lemma}
\newtheorem{coro}{Corollary}

\newtheorem{exam}{Example}
\newtheorem{rem}{Remark}

\newtheorem{definition}{Definition}

\title{On Some Aspects of Pseudonorm Compact and Montel Operators on Locally Solid Vector Lattices}

\author{Nazife Erkur{\c s}un-{\"O}zcan$^{1,*},$ Niyazi An{\i}l Gezer$^2$\\
	{\small $^{1}$ Department of Mathematics,} \\ {\small Hacettepe University, 06800 Ankara, Turkey.} \\{\small erkursun.ozcan@hacettepe.edu.tr} \\
	{\small $^{2}$ Department of Mathematics,} \\ {\small Middle East Technical University, 06800 Ankara, Turkey.} \\ {\small ngezer@metu.edu.tr} }

\begin{document}

\maketitle
\noindent

\begin{abstract}
Unbounded convergences have been applied successfully to locally solid topologies on vector lattices. In the present paper, we first expose several properties of various classes of Riesz pseudonorms on vector lattices. We accomplish this by abstracting some generalities of the norm of an $AM$-space with strong norm unit to locally solid topologies induced by a pseudonorm. By using these classes of pseudonorms, we study compactness properties of operators (not necessarily linear) between locally solid (not necessarily Hausdorff) topologies. We study new classes of operators such as pseudonorm compact, pseudo-semicompact and pseudo-$AM$-compact operators as well as the classical Montel operators. 



\end{abstract} 

\maketitle

\section{Introduction}
One of the main goals of this paper is to study compactness properties of possibly non-linear operators between possibly non-Hausdorff locally solid vector lattices. We accomplish this by abstracting some generalities of the norm of an $AM$-space with a strong norm unit to locally solid topologies induced by a pseudonorm. We then study compactness properties of operators between such spaces by utilizing the properties of these pseudonorms. Several classes and examples of compact-like operators with completely new origins are given. Comparison of various versions of boundedness forms the basic method of our analyses throughout. Most of our results are motivated from~\cite{EV2017,AEEM2018, DEM2018}. 

Recent manuscripts, see~\cite{DE2018, EV2017, GXU, GTX, LC, Y, W1977, EG2019, Z2018}, about various types of convergences on vector lattices, and, the recent progress towards the operator theory on both vector lattice valued pseudonormed spaces and on spaces equipped with unbounded convergences, see~\cite{AEEM2018, DE2018, DEM2018, EG2019, EGZ2018, EGZ2019, Pli2011, PliPo} and the references therein, stimulated us to write this article. Some of the present material can also be profitably used in the settings of vector lattice valued pseudonormed spaces.

Structure of the paper is as follows. In Section~\ref{PseudoBasics}, we study various notions related to pseudonormed vector lattices. Most of these notions are motivated from the factorization theory given in the sequel. We derive many results related to \textit{bounded}, \textit{almost bounded}, \textit{coarsely almost bounded} and \textit{semibounded pseudonorms}, see Definition~\ref{Def001}. Historically speaking, pseudonorms satisfying certain order theoretical properties were studied by many. For instance, Fatou and order continuous pseudonorms are among the most useful pseudonorms. In addition to the fact that main results of \cite{EV2017} establish the relationships between unbounded convergences and pseudonorms on vector lattices, results presented in Section~\ref{PseudoBasics} abstract further properties of these special pseudonorms which can be combined with the settings of~\cite{DE2018}. In Section~\ref{PseudoCompact}, we study compactness properties of operators between two pseudonormed vector lattices. Taken all together, results presented in Section~\ref{PseudoCompact} form a factorization theory, see~\cite[Chapter 18]{Z1983} and~\cite{AEEM2018}, of operators having compactness properties. We derive various results related to \textit{pseudonorm compact}, \textit{pseudo-semicompact} and \textit{pseudo-$AM$-compact} operators, see Definition~\ref{DefnPsCompact} and Definition~\ref{DefnSemicompactVer}. In the last part of Section~\ref{PseudoCompact}, we study $M$-weakly compact and $L$-weakly compact operators on pseudonormed vector lattices. In Section~\ref{Fatou}, we study compact and Montel operators on locally solid Hausdorff spaces. By~\cite[Theorem 2.28]{AB1}, a linear topology $\tau$ is locally solid if and only if there exists a family $\{\rho_i \}_{i\in I}$ consisting of $\tau$-continuous Riesz pseudonorms generating $\tau$. Hence, locally solid vector lattices arising in Section~\ref{Fatou} have the common property that their topology is induced by a system of pseudonorms. The notion of Montel operators can be seen as an abstraction of pseudonorm compact operators.
\section{Preliminaries}

Unexplained terminology about vector lattices can be found in ~\cite{AB2, F1974, LZ1971, MN91, V1967, Z1983}. Throughout this paper, all vector lattices are assumed to be Archimedean. 

Let $X$ be a vector lattice. We say that a net $x_{\alpha}$ in $X$ order converges to $x\in X$ if there exists a net $y_{\beta},$ possibly over a different index set, such that $y_{\beta}\downarrow 0$ and that for every $\beta$ there exists some $\alpha_0$ such that $|x_{\alpha}-x|\leq y_{\beta}$ for all $\alpha_0\leq \alpha$. In this case, we write $x_{\alpha}\xrightarrow{o} x$. A net $x_{\alpha}$ in $X$ \textit{unbounded order convergent} to $x\in X$ if $|x_{\alpha}-x|\wedge u\xrightarrow{o} 0$ for all $u\in X^+$. In this case, we say that $x_{\alpha}$ $uo$-converges to $x,$ and, we write $x_{\alpha}\xrightarrow{uo} x.$ A net $x_{\alpha}$ in a normed lattice $X$ \textit{unbounded norm converges} to $x$ if $|x_{\alpha}-x|\wedge u\xrightarrow{\norm{\cdot}} x$ for every $u\in X^+$. In this case, we write $x_{\alpha}\xrightarrow{un} x$. A systematic study of unbounded convergences can be found in~\cite{Y,DEM2018, EV2017,GXU,GTX,LC, W1977}. In general, unbounded order convergence is not topological. However, unbounded norm convergence defines a topology on $X,$ more details can be found in the aforementioned articles and and in the references therein.
\begin{exam}
	The notion of unbounded order convergence in vector lattices is a generalization of almost everywhere convergence. Let $(\Omega,\Sigma,\mu)$ be a $\sigma$-finite measure space. A sequence $x_n$ in $L^p(\Omega)$ order converges to $x\in L^p(\Omega)$ $(1\leq p \leq \infty)$ if and only if $x_n$ converges to $x$ almost everywhere and there exists some $z\in L^p(\Omega)$ such that $|x_n|\leq z$ for all $n$. In the case $p<\infty$, $x_n$ is unbounded order convergent to $x$ if and only if $x_n$ converges almost everywhere to $x$. In the cases of $c_0$ and $\ell_p$ with $1\leq p\leq \infty,$ $uo$-convergence of nets agrees with coordinate-wise convergence, see~\cite{GXU}. 
\end{exam}

Vector lattices are prime examples, see~\cite[page 36]{F1974}, of non-locally convex spaces in analysis. A linear topology on a vector lattice $X$ is said to be \textit{locally solid} if zero has a neighborhood basis consisting of solid sets. A subset $B$ of $X$ is said to be \textit{topologically bounded} if for each zero neighborhood $U$ in $X$ there exists a $\lambda>0$ such that $B\subseteq \lambda U.$ A locally solid vector lattice $X$ is said to have \textit{Lebesgue property}, see~\cite[Chapter 3]{AB1} and~\cite[Section 2.24]{F1974}, if every order null net in $X$ converge to zero with respect to the locally solid topology on $X$. A locally solid topology on a vector lattice $X$ can be used to construct a new unbounded locally solid topology on $X$. We refer the reader to~\cite{AEEM2018, DE2018, DEM2018,EV2017} for more details.

Let $X$ be a vector lattice. A non-negative function $\rho\colon X\to \rr$ is said to be a pseudonorm on $X$ if $\rho(x+y)\leq \rho(x)+\rho(y)$ and $\lim_{\theta\to 0} \rho(\theta x)=0$ for all $x,y\in X$. A pseudonorm $\rho$ is said to be a \textit{Riesz pseudonorm} if $|x|\leq |y|$ implies $\rho(x)\leq \rho(y)$ for $x,y\in X$. Although we use the classical definitions of normed vector lattice and seminormed vector lattice, see~\cite{AB2,MN91,Z1983}, a more general definition for pseudonormed vector lattices is required for our purposes. More details are given in Section~\ref{PseudoBasics}.

\section{Unbounded Pseudonorms on Vector Lattices}\label{PseudoBasics}
By a pseudonormed vector lattice $(X,\rho)$ we mean a pair consisting of an Archimedean vector lattice $X$ together with a pseudonorm $\rho\colon X\to \rr$. For the purposes of the present paper, it is more convenient to separately write the Riesz pseudonorm property whenever it is needed.

We put $U_{\rho}\coloneqq \{x\in X\colon \rho(x)\leq 1\}$ for the unit ball with respect to $\rho$. If the pseudonorm in discussion is clear from the context, we denote by $X$ the pseudonormed vector lattice $(X,\rho)$.
\begin{exam}\label{Example0001}
	Let $X$ be a vector lattice. It follows from classical definitions that every seminorm on $X$ is a pseudonorm. If $\rho$ is a Riesz seminorm on $X$ then $\rho$ is a Riesz pseudonorm. If $\rho$ is a seminorm then there are some derived expressions such as $x\mapsto \frac{\rho(x)}{1+\rho(x)},$ $x\mapsto \frac{\rho(x)}{2+\rho(x)},$ and so forth; which are pseudonorms but not seminorms, in general. Further, the assignment $x\mapsto \theta\frac{\rho(x)}{1+\rho(x)}+(1-\theta)\frac{\rho(x)}{2+\rho(x)}$ for $\theta\in [0,1]$ is again a pseudonorm. If $f$ is a functional on $X$ then $\rho_f(x)=|f(x)|$ is a seminorm on $X$. If the functional $f\colon X\to \rr$ is order bounded then $\rho_f(x)=|f(|x|)|$ is a Riesz seminorm on $X$. It is also clear that not every pseudonorm is a vector (lattice) norm, see~\cite{AEEM2018, DE2018}, on $X$ and vice versa.
\end{exam}
The observation given in the next example demonstrates an important difference between normed and pseudonormed cases.
\begin{exam}\label{Example0002}
	Let $(X,\rho)$ be a pseudonormed vector lattice and $\epsilon>0$. In general, it does not follow from $x\in \epsilon U_{\rho}$ that $\rho(x)\leq \epsilon$. In view of Example~\ref{Example0001}, consider the pseudonorm $\rho(x)=|x|/(1+|x|)$ on the set of real numbers. For $\epsilon=1/2$ and $x=-3/2$ one has $x\in \epsilon U_{\rho},$ but $\rho(x)>\epsilon$. 
\end{exam}

\begin{exam}
	Let $X$ be an atomic normed vector lattice. Denote by $\{x_{\alpha}\}_{\alpha}\subseteq X^+$ the corresponding orthogonal maximal system of atoms. An element $x\in X$ can uniquely be written as $x=o{\text -}\sum_{\alpha}\lambda_{\alpha}x_{\alpha}$ for $\lambda_{\alpha}\in \rr.$ We recall that order convergence is used in this series representation. Hence, one can derive various pseudonorms, as done in Example~\ref{Example0001}, by using the norm continuous functionals $\phi_{\alpha}(x)=\lambda_{\alpha}$ on $X$.
\end{exam}

A nonempty subset $B\subseteq X$ of a pseudonormed vector lattice $(X,\rho)$ is said to be \textit{pseudo-bounded} if the set $\rho(B)$ is bounded. If $\rho$ is a Riesz pseudonorm on $X$ and $B\subseteq X$ is order bounded then $B$ is pseudo-bounded. In this case, i.e., when $\rho$ is a Riesz pseudonorm on $X$, a set $B\subseteq X$ is pseudo-bounded if and only if the set $|B|=\{|b|\colon b\in B \}\subseteq X^+$ is pseudo-bounded.

\begin{exam}
	Let $K$ be an infinite compact Hausdorff space, $x_0\in K$. As usual, $C(K)$ denotes the Banach lattice of real valued continuous functions on $K$. We denote by $\delta_{x_0}$ the evaluation functional at $x_0\in K,$ i.e., $\delta_{x_0}(f)=f(x_0)$ for $f\in C(K).$ We put $\rho(f)=|\delta_{x_0}(|f|)|=|f|(x_0)$ for $f\in C(K).$ By considering the set of continuous functions vanishing at $x_0$ we conclude that not every pseudo-bounded set is order bounded in $C(K)$. We further observe that in general $\rho(\inf_{\alpha}f_{\alpha})\neq \inf_{\alpha} \rho(f_{\alpha})$ even in the case $f_{\alpha}\downarrow 0$ in $C(K)$.
\end{exam}

Although the notion of almost order bounded sets in vector lattices is classical, see~\cite{AB2, LZ1971,Z1983}, a conceptional notion for almost pseudo-bounded sets is needed in the factorization theory given in Section~\ref{PseudoCompact}. 

A nonempty subset $B$ of the pseudonormed vector lattice $(X,\rho)$ is said to be \textit{almost pseudo-bounded} if for every $\epsilon>0$ there exists an order interval $[x,y]$ in $X$ such that $B\subseteq [x,y]+\epsilon U_{\rho}.$ In the case of seminorms, this definition agrees with~\cite[Proposition 2.3.2]{MN91}.

\begin{exam}
	Let $\rho$ be a Riesz pseudonorm on $X$. If a subset $B\subseteq X$ is order bounded then $B$ is almost pseudo-bounded. Every almost pseudo-bounded set is pseudo-bounded. 
\end{exam}

Following characterization is standard in the case of the lattice norm of a Banach lattice, see~\cite[Theorem 121.1]{Z1983}. In view of Example~\ref{Example0002}, we have a similar characterization in the case of Riesz seminorms.

\begin{lem}\label{Ppseudoalmostbounded}
	Suppose that $\rho$ is a Riesz seminorm on a vector lattice $X$. Following statements are equivalent for a nonempty subset $B$ of $X$:
	\begin{itemize}
		\item[\em i.] {The set $B$ is almost pseudo-bounded.} 
		\item[\em ii.] {For every $\epsilon>0$ there exists $u\in X^+$ such that $\rho((|x|-u)^+)\leq \epsilon$ for all $x\in B$. }
	\end{itemize}
\end{lem}
\begin{proof}
$(i)\Rightarrow (ii)$. If $B$ is almost pseudo-bounded, then given $\epsilon>0$ there is $[x',y']\subseteq X$ such that $B\subseteq [x',y']+\epsilon U_{\rho}.$ Let $u=|x'|\vee |y'|.$ It follows that $|x|\in [-u,u]+\epsilon U_{\rho}$ for every $x\in B$. Hence, $|B|$ is also almost pseudo-bounded. Hence, $|x|=x_1+x_2$ where $x_1\in [-u,u]$ and $x_2\in \epsilon U_{\rho}.$ Also, $|x|=|x|\wedge u+(|x|-u)^+.$ Hence, from \[(|x|-u)^+=|x|-|x|\wedge u\leq (|x|-x_1)^+=x_2^+\leq |x_2|\] it follows that $\rho((|x|-u)^+)=\rho(|x|-|x|\wedge u)\leq \rho(x_2)\leq \epsilon$. Indeed, since $\rho$ is a Riesz seminorm, $x_2\in \epsilon U_{\rho}$ implies $\rho(x_2)\leq \epsilon$ for arbitrary $\epsilon>0$. 

$(ii)\Rightarrow (i)$. Given $\epsilon>0$ let $u\in X^+$ be such that $\rho((|x|-u)^+)\leq \epsilon$ for every $x\in B$. Since $\rho$ is a Riesz seminorm, $(|x|-u)^+\in \epsilon U_{\rho}$ for every $x\in B$. Hence, $|x|=|x|\wedge u+(|x|-u)^+$ implies that $|B|\subseteq [-u,u]+\epsilon U_{\rho}.$ From $0\leq x^+\leq u+(|x|-u)^+$ and $0\leq x^-\leq u+(|x|-u)^+$ we conclude that there exists $x_1,x_2,x_3,x_4$ satisfying $x^+=x_1+x_2,$ $x^-=x_3+x_4,$ $x_1-x_3\in [-u,u]$ and $x_2-x_4\in [-(|x|-u)^+,(|x|-u)^+].$ Hence, $x=(x_1-x_3)+(x_2-x_4)$ with $x_1-x_3\in [-u,u]$ and $\rho(x_2-x_4)\leq 2\rho((|x|-u)^+)$ imply that the set $B$ is almost pseudo-bounded.
\end{proof}

\begin{prop}
	Let $X$ be a vector lattice, and, consider the the canonical embedding $\widehat{} \colon X\to X^{\sim \sim}$ of $X$ into its second order dual $X^{\sim \sim}$. Let $x_1\leq x_2$ in $X.$ If a set $B\subseteq X^{\sim}$ is almost pseudo-bounded with respect to $f\mapsto |\hat{x}_1(|f|)|$ then it is also almost pseudo-bounded with respect to $f\mapsto |\hat{x}_2(|f|)|.$ 
\end{prop}
\begin{proof}
	Let $f_1,f_2$ be order bounded functionals on $X$ such that $f_1\leq f_2$. Let $\rho_i(x)=|f_i(|x|)|$ for $x\in X$ and $i=1,2.$ We first show that if a nonempty subset $B\subseteq X$ is almost pseudo-bounded with respect to $\rho_2$ then it is almost pseudo-bounded with respect to $\rho_1$. It follows from $f_1\leq f_2$ that $\rho_1(x)=|f_1(|x|)|\leq |f_2(|x|)|=\rho_2(x)$ for every $x\in X$. Hence, $U_{\rho_2}=\{x\colon \rho_2(x)\leq 1 \}\subseteq \{x\colon \rho_1(x)\leq 1 \}=U_{\rho_1}.$ Suppose that a set $B\subseteq X$ is almost pseudo-bounded with respect to $\rho_2$. For every $\epsilon>0$ there exists an order interval $[x,y]\subseteq X$ such that $B\subseteq [x,y]+\epsilon U_{\rho_2}.$ It follows that $B\subseteq [x,y]+\epsilon U_{\rho_1}.$ Hence, $B$ is almost pseudo-bounded with respect to $\rho_1$.
	
	It is a classical fact, see~\cite[page 61]{AB2}, that the canonical embedding $\widehat{} \colon X\to X^{\sim \sim}$ of a vector lattice $X$ into its second order dual $X^{\sim \sim}$ is lattice preserving. By the above paragraph, we conclude that in the case $x_1\leq x_2$ in $X$ if a set $B\subseteq X^{\sim}$ is almost pseudo-bounded with respect to $f\mapsto |\hat{x}_1(|f|)|$ then it is also almost pseudo-bounded with respect to $f\mapsto |\hat{x}_2(|f|)|.$
\end{proof}

Let $(X,\rho)$ be a pseudonormed vector lattice. We write $\varrho(x,y)\coloneqq\rho(|x-y|)$ for the pseudo-metric $\varrho$ on $X$ induced by $\rho$. Many properties of this pseudometric follow from that of absolute value in vector lattices, see~\cite{AB2,LZ1971,V1967}. For instance, $\varrho(x,y)=\varrho(y,x)$ for every $x,y\in X.$ 

A net $x_{\alpha}$ in $X$ pseudonorm converges to some $x,$ denoted by $x_{\alpha}\xrightarrow{\rho} x,$ if and only if $\rho(|x_\alpha-x|)\xrightarrow{} 0.$ Evidently, the set $U_{\rho}=\{x\in X\colon \rho(x)\leq 1 \}$ is pseudonorm closed.

\begin{prop}\label{pseudoMetricProp}
	Let $\rho$ be a Riesz pseudonorm on a vector lattice $X$. For every $x,y$ and $z$ in $X$ we have $\varrho(x\vee z,y\vee z)\leq \varrho (x,y),$ $\varrho(x\wedge z,y\wedge z)\leq \varrho(x,y)$ and $\varrho(x^{\pm},y^{\pm})\leq \varrho(x,y).$
\end{prop}
\begin{proof}
	For $x,y$ and $z$ in $X$ we have $|x\vee z-y\vee z|\leq |x-y|,$ $|x\wedge z-y\wedge z|\leq |x-y|,$ and, $|x^{\pm}-y^{\pm}|\leq |x-y|.$ Since $\rho$ is a Riesz pseudonorm, the statements of the proposition follow from these inequalities.
\end{proof}
Let $\rho$ be a Riesz pseudonorm on a vector lattice $X$. We remark that the locally solid topology on $X$ induced by a pseudonorm $\rho$ is not necessarily Hausdorff. It is quite possible in the general settings that $x_{\alpha}\xrightarrow{\rho} x$ and $x_{\alpha}\xrightarrow{\rho} y$ with $x\neq y.$

We recall from~\cite[Chapter 3.11]{V1967} that a sequence $x_n$ in a vector lattice $X$ is said to converge relatively uniformly to $x\in X,$ denoted by $x_n\xrightarrow{r} x,$ if there exists $u\in X^+$ such that for every $\epsilon>0$ there exists an $N$ such that $|x_n-x|\leq \epsilon u$ for $n\geq N.$ 
\begin{prop}\label{RelativeUniformToPseudo}
	Let $\rho$ be a Riesz pseudonorm on a vector lattice $X$. If $x_n\xrightarrow{r}x$ for a sequence $x_n$ in $X$ then $x_n\xrightarrow{} x$ with respect to the pseudo-metric $\varrho$.
\end{prop}
\begin{proof}
	If $x_n\xrightarrow{r}x$ then there is some $u\in X$ and a sequence $\lambda_n$ decreasing to $0$ such that $|x_n-x|\leq \lambda_nu$ for all $n=1,2,\ldots.$ As $\rho$ is a Riesz pseudonorm, $\rho(|x_n-x|)=\varrho(x_n,x)\leq \rho(\lambda_n u)$ for each $n$. This shows that $x_n\xrightarrow{}x$ with respect to the pseudo-metric $\varrho$.
\end{proof}
\begin{rem}
	We recall from~\cite[Theorem 84.3]{LZ1971} and~\cite[Chapter VI.4]{V1967} that order convergence on a vector lattice $X$ is said to be stable if for every $x_n\xrightarrow{o} 0$ there exists $0\leq \lambda_n\uparrow \infty$ such that $\lambda_nx_n\xrightarrow{o}0$. Order convergence on $L^p$ and $\ell_p$ is stable. Almost regular vector lattices are among those vector lattices satisfying this stability assumption. If $X$ is a $\sigma$-order complete vector lattice with the diagonal property, sequential order convergence is equivalent to the relative uniform convergence. In these cases, it follows from Proposition~\ref{RelativeUniformToPseudo} that $x_n\xrightarrow{o} x$ implies $x_n\xrightarrow{\varrho} x$.
\end{rem}

\begin{lem}\label{L001}
	Suppose that $\rho$ is a Riesz pseudonorm on a vector lattice $X$. Let $x_{\alpha}$ and $y_{\beta}$ be two nets in $X$ such that $x_{\alpha}\xrightarrow{\rho} x$ and $y_{\beta}\xrightarrow{\rho} y$. Then $x_{\alpha}\wedge y_{\beta}\xrightarrow{\rho} x\wedge y.$ In particular, if $x_{\alpha}\xrightarrow{\rho} x$ then $x_{\alpha}^+\xrightarrow{\rho} x^+$.
\end{lem}
\begin{proof}
	The topology on $X$ generated by the Riesz pseudonorm $\rho$ is a locally solid topology. It follows from \cite[Theorem~2.17]{AB1} that the lattice operations $(x,y)\mapsto x\wedge y$ and $x\mapsto x^+$ are uniformly continuous. 
\end{proof}

\begin{coro}\label{BandvsRiesznorm}
	Suppose that $\rho$ is a Riesz norm on a vector lattice $X$. If $B$ is a band in $X$ then $B$ is closed with respect to the locally solid topology induced by $\rho$.
\end{coro}
\begin{proof}
	By Lemma~\ref{L001}, lattice operations on $X$ are continuous with respect to the locally solid topology induced by $\rho$. Let $x_{\alpha}$ be a net in $B$ such that $x_{\alpha}\xrightarrow{\rho} x$ for some $x\in X.$ For every $y\in B^{\perp}$ we have $|x_{\alpha}|\wedge |y|\xrightarrow{\rho}|x|\wedge |y|.$ Since $x_\alpha\in B,$ we have $|x_{\alpha}|\wedge |y|=0$ for all $\alpha.$ Since $\rho$ is a Riesz norm, it follows that $|x|\wedge |y|=0$. Hence, $x\in B$.
\end{proof}

\begin{exam}
	Conclusion of Corollary~\ref{BandvsRiesznorm} does not hold if one considers more general cases. Let $K$ be a compact Hausdorff space, $F\subseteq K$ be a closed set which is the closure of some open subset of $K,$ and $x\in K\backslash F$. We denote by $\rho(f)=|f|(x)$ the Riesz pseudonorm induced by the evaluation functional at $\delta_x(f)=f(x)$ for $f\in C(K).$ The band $B$ of functions vanishing on $F,$ i.e., the band $B=\{f\in C(K)\colon f|_F\equiv 0 \},$ is in general not closed with respect to locally solid topology on $C(K)$ induced by $\rho.$ 
\end{exam}

\begin{exam}
	When $\rho$ is a Riesz pseudonorm, we have $x\in U_{\rho}$ if and only if $|x|\in U_{\rho}$. Although lattice operations are continuous with respect to the locally solid topology induced by the Riesz pseudonorm $\rho$, see Lemma~\ref{L001}, the generating cone $X^+$ of the vector lattice $X$ is not closed, in general. 
\end{exam}

Let $(X,\rho)$ be a pseudonormed vector lattice. A subset $B$ of $X$ is said to be \textit{conditionally pseudonorm compact} if for every net $x_{\alpha}$ in $B$ there exists a subnet $x_{\alpha_\beta}$ and $x\in X$ such that $x_{\alpha_{\beta}}\xrightarrow{\rho} x$. The set $B$ is conditionally pseudonorm compact if and only if it is conditionally compact with respect to the locally solid topology induced by $\rho$. It follows that if $B$ is conditionally pseudonorm compact then $\rho(B)$ is totally bounded. If a set is conditionally pseudonorm compact then it is pseudo-bounded.
\begin{exam}
	Let $(X,\norm{\cdot})$ be a Banach lattice. It readily follows that a subset of $X$ is precompact with respect to norm $\norm{\cdot}$ if and only if it is conditionally pseudonorm compact with respect to norm $\norm{\cdot}$.
\end{exam}
Following definitions are motivated from the factorization theory of compact operators, and, they are needed in the sequel, see Section~\ref{PseudoCompact}. Items $v,vi$ and $vii$ of Definition~\ref{Def001} are quite standard and well-motivated. 
\begin{definition}\label{Def001} A pseudonorm $\rho$ on a vector lattice $X$ is said to be 
	\begin{itemize}
	\item[\em i.] {a bounded pseudonorm, if every pseudo-bounded set is order bounded in $X$. }
	\item[\em ii.] {an almost bounded pseudonorm, if every almost pseudo-bounded set is order bounded in $X$. }
	\item[\em iii.] {a coarsely almost bounded pseudonorm, if $(X,\norm{\cdot})$ is a normed lattice and every almost order bounded subset of $(X,\norm{\cdot})$ is almost pseudo-bounded in $X$.}
	\item[\em iv.] {a semibounded pseudonorm, if every pseudo-bounded set is almost pseudo-bounded in $X$.}
	\item[\em v.] {finite type, if $U_\rho=\{x\in X\colon \rho(x)\leq 1 \}$ is conditionally pseudonorm compact in $X$.}	
	\item[\em vi.] {order continuous, see~\cite[Definition 4.7]{AB2}, if $x_{\alpha}\xrightarrow{o} x$ implies $x_{\alpha}\xrightarrow{\rho} x$ in $X$. }
	\item[\em vii.] {coarser (finer) than $\norm{\cdot}$, if $\norm{\cdot}$ is a norm on $X$ and $\rho(x)\leq \norm{x}$ (respectively, $\norm{x}\leq \rho(x)$) for every $x\in X^+$.}
 \end{itemize}
\end{definition}
A norm or seminorm on $X$ is said to satisfy some of the properties given in Definition~\ref{Def001}, if it is a pseudonorm, which is a norm or seminorm respectively, satisfying these properties. 

\begin{exam}\label{ExampleAMFirst}
	We recall that a vector $e\in X^+$ in a normed vector lattice $(X,\norm{\cdot})$ is said to be a strong norm unit, see~\cite[page 455]{Z1983}, if $\norm{e}=1$ and $\norm{x}\leq 1$ for $x\in X$ implies $|x|\leq e$. Let $(X,\norm{\cdot})$ be an $AM$-space with a strong norm unit. Then the norm $\norm{\cdot}$ is bounded in the sense of Definition~\ref{Def001}, because norm bounded subsets of $X$ are order bounded. By~\cite[Theorem 122.2]{Z1983}, it further follows in this case that the norm $\norm{\cdot}$ is also almost bounded and semibounded. Let $K$ be a compact Hausdorff topological space. The Banach lattice $C(K)$ is a concrete example of $AM$-space with a strong norm unit. 
\end{exam}

\begin{lem}\label{CoarselyLemma}
	Let $\rho$ be a pseudonorm on a normed lattice $(X,\norm{\cdot}).$ If $U_{\rho}=U_{\norm{\cdot}}$ then $\rho$ is a coarsely almost bounded pseudonorm on $X$. 
\end{lem}
\begin{proof}
	Since $U_{\rho}=U_{\norm{\cdot}},$ a subset $B$ of $X$ is almost order bounded if and only if it is almost pseudo-bounded. 
\end{proof}

\begin{exam}
	It follows that any lattice norm on a vector lattice is a coarsely almost bounded pseudonorm. This is the case because the notions of almost order bounded sets and order bounded sets are common notions. In view of Lemma~\ref{CoarselyLemma}, there are coarsely almost bounded pseudonorms which are not norms. 
\end{exam}

\begin{exam}
	Consider the classical Banach lattice $L^p(X)$ for $1\leq p<\infty$ where $X$ is a $\sigma$-finite measure space. It follows from monotone convergence theorem that for each decreasing non-negative sequence $f_n$ in $L^p(X)$ one has $\inf_n f_n=0$ implies $\norm{f_n}_p\xrightarrow{} 0$. Hence the spaces $L^p(X)$ for $1\leq p<\infty$ have order continuous norm.
\end{exam}
\begin{definition}\label{Def002} A pseudonorm $\rho$ on a vector lattice $X$ is said to have
	\begin{itemize}
		\item[\em i.] {ideal (band) property, if the order ideal (respectively, band) generated by $U_{\rho}$ is equal to $X$. } 
		\item[\em ii.] {Fatou property, if $\emptyset \neq B\uparrow x$ in $X^+$ implies $\rho(x)=\sup\{\rho(y)\colon y\in B \}.$ }
		\item[\em iii.] { subsequence property, if $\rho(x_{n})\xrightarrow{} 0$ implies $x_{n_{k}}\xrightarrow{o} 0$ for some subsequence $x_{n_{k}}$ of $x_{n}$. If the analogous statement holds for nets, then we say that $\rho$ has the subnet property.}
	\end{itemize}
\end{definition}
A norm or seminorm on $X$ is said to satisfy some of the properties given in Definition~\ref{Def002}, if it is a pseudonorm, which is a norm or seminorm respectively, satisfying these properties. 
\begin{exam}
	Let $K$ be an infinite compact Hausdorff space, $x_0\in K$. Let $\delta_{x_0}$ denote the evaluation functional on $C(K)$ at $x_0\in K,$ i.e., $\delta_{x_0}(f)=f(x_0)$ for $f\in C(K).$ We put $\rho(f)=|\delta_{x_0}(|f|)|=|f|(x_0)$ for $f\in C(K).$ It follows that the Riesz pseudonorm $\rho$ does not have the band property, and hence, it does not have the ideal property. The classical uniform norm on $C(K)$ has the Fatou property but it is not order continuous.
\end{exam}

\begin{exam}\label{AnExample}
	Let $(X, \norm{\cdot})$ be a Banach lattice. By~\cite[Theorem VII.2.1]{V1967}, the norm $\norm{\cdot}$ on $X$ has the subsequence property.
\end{exam}

\begin{exam}
	Let $\rho$ be an order continuous pseudonorm on a vector lattice $X$. Let $(Y,\norm{\cdot})$ be a normed vector lattice. If $T\colon (X,\rho)\to (Y,\norm{\cdot})$ is continuous then $T\colon X\to Y$ is order-to-norm continuous, i.e., $x_{\alpha}\xrightarrow{o}x$ implies $Tx_{\alpha}\xrightarrow{\norm{\cdot}} x$. Indeed, if $x_{\alpha}\xrightarrow{o}0$ then $x_{\alpha}\xrightarrow{\rho} 0$ as $\rho$ is order continuous. Hence, $\norm{Tx_{\alpha}}\xrightarrow{}0.$ 
\end{exam}

\begin{prop}
	A Riesz pseudonorm $\rho$ on a vector lattice $X$ is bounded if and only if every pseudo-bounded subset of $X^+$ is order bounded.
\end{prop}
\begin{proof}
	Forward direction follows from Definition~\ref{Def001}. Conversely, let $B\subseteq X$ be a pseudo-bounded subset. As $\rho$ is a Riesz pseudonorm, $|B|\subset X^+$ is also pseudo-bounded. It follows that $|B|$ is order bounded. This implies that $B\subset X$ is order bounded. Hence, $\rho$ is a bounded pseudonorm on $X$.
\end{proof}

\begin{prop}
	A Riesz pseudonorm $\rho$ on a vector lattice $X$ is almost bounded if and only if every almost pseudo-bounded subset of $X^+$ is order bounded.
\end{prop}
\begin{proof}
	Forward direction follows from Definition~\ref{Def001}. Conversely, let $B\subseteq X$ be an almost pseudo-bounded subset. As $\rho$ is a Riesz pseudonorm, $|B|\subset X^+$ is also almost pseudo-bounded. It follows that $|B|$ is order bounded. This implies that $B\subset X$ is order bounded. Hence, $\rho$ is an almost bounded pseudonorm on $X$.
\end{proof}
\begin{prop}\label{PseudoConvex}
	The collection of all bounded pseudonorms on a vector lattice $X$ is a convex set. Similarly, the collection of all finite type seminorms and the collection of all order continuous seminorms on $X$ are convex sets.
\end{prop}

\begin{proof}
	Evidently, this collection is non-empty. Let $\theta\in [0,1]$ and $\rho_1,\rho_2$ be two bounded pseudonorms on $X$. For a set $B\subseteq X,$ the set $(\theta \rho_1+(1-\theta)\rho_2)(B)$ is bounded if and only if $\rho_1(B)$ and $\rho_2(B)$ are bounded. Since $\rho_1$ and $\rho_2$ are bounded pseudonorms, if $(\theta \rho_1+(1-\theta)\rho_2)(B)$ is bounded then the set $B$ is order bounded in $X$. This shows that a convex combination of two bounded pseudonorms is again a bounded pseudonorm.
	
	Let $\rho_1,\rho_2$ be two finite type seminorms on $X$. Let $\theta\in [0,1]$ be fixed. For every net $x_{\alpha}\in U_{\theta\rho_1+(1-\theta)\rho_2}$ there exists $N_1,N_2$ such that $\rho_1(x_{\alpha})\leq N_1$ and $\rho_2(x_{\alpha})\leq N_2$ for all $\alpha$. As $\rho_1$ and $\rho_2$ are finite type seminorms, the sets $N_1U_{\rho_1}$ and $N_2U_{\rho_2}$ are conditionally pseudonorm compact with respect to $\rho_1$ and $\rho_2,$ respectively. Hence, there is a subnet $x_{\alpha_{\beta_{\gamma}}}$ and $x',y'\in X$ such that $x_{\alpha_{\beta_{\gamma}}}\xrightarrow{\rho_1}x'$ and $x_{\alpha_{\beta_{\gamma}}}\xrightarrow{\rho_2}y'.$ Hence, $x_{\alpha_{\beta_{\gamma}}}\xrightarrow{\theta\rho_1+(1-\theta)\rho_2} \theta x'+(1-\theta) y'$. This shows that $ U_{\theta\rho_1+(1-\theta)\rho_2}$ is conditionally pseudonorm compact with respect to $\theta\rho_1+(1-\theta)\rho_2$. In view of Definition~\ref{Def001}, $\theta \rho_1+(1-\theta)\rho_2$ is finite type.
	
	For the last statement, let $\rho_1,\rho_2$ be order continuous pseudonorms on $X$. Let $\theta\in [0,1]$ be fixed. If a net $x_{\alpha}$ in $X$ satisfies $x_{\alpha}\xrightarrow{o} x$ for some $x\in X$ then $x_{\alpha}\xrightarrow{\rho_1}x$ and $x_{\alpha}\xrightarrow{\rho_2} x$ hold. Evidently, $x_{\alpha}\xrightarrow{\theta\rho_1+(1-\theta)\rho_2}x$ holds. Hence, $\theta \rho_1+(1-\theta)\rho_2$ is an order continuous seminorm.
\end{proof}
\begin{rem}
	In view of Proposition~\ref{PseudoConvex}, it is natural to ask if the collection of all almost bounded pseudonorms on a vector lattice $X$ is closed under convex combinations. The answer to this question is negative. Let us consider a more specific question. Suppose that $\rho_1$ and $\rho_2$ are Riesz pseudonorms on $X$ having the same almost pseudo-bounded sets. If $\rho_1$ and $\rho_2$ are almost bounded pseudonorms then is it true that $\theta \rho_1+(1-\theta) \rho_2$ again an almost bounded pseudonorm? In this case, it follows from $\rho_1(x)\leq \frac{1}{n}\rho_1(\frac{x}{n})$ for $n\in \nn$ that $\{nx\colon \rho(x)\leq 1 \}\subseteq \{x\colon \rho(x)\leq n\}$ and the question reduces to almost pseudo-boundedness of $\{x\colon \rho(x)\leq n\}$.
\end{rem}

Following result is motivated from~\cite[Theorem 2.1]{EV2017}. It provides an analytical technique to derive a Riesz pseudonorm from a Riesz seminorm. The resulting Riesz pseudonorm is known to have topological relationships with the initial Riesz seminorm.
\begin{lem}\label{UnboundedRiesz}
	Let $u\in X^+$ be arbitrary. If $\rho$ is a Riesz seminorm on a vector lattice $X$ then $\rho_u(x)\coloneqq \rho(|x|\wedge u)$ is a Riesz pseudonorm on $X$. 
\end{lem}
\begin{proof} It is clear that $\rho_u(x)\geq 0$ for all $x\in X$ and that $\rho_u(0)=0$. Since $\rho$ is a Riesz seminorm, it follows from $|x+y|\leq |x|+|y|$ for $x,y\in X$ that $\rho_u(x+y)\leq \rho_u(x)+\rho_u(y).$ Let $\lambda_n$ be a sequence of real numbers such that $\lambda_n\xrightarrow{} 0.$ It follows from \[\rho_u(\lambda_nx)=\rho(|\lambda_n x|\wedge u)\leq |\lambda_n|\rho(x)\] that $\rho_u(\lambda_n x)\xrightarrow{}0$ for all $x\in X.$ Hence, $\rho_u$ is a Riesz pseudonorm on $X$.
\end{proof}

\begin{exam}
	Consider $X=C[0,1],$ the vector lattice of continuous real valued functions on $[0,1]$ with the pointwise ordering. Let $u\in C[0,1]$ be positive. For each $1\leq p < \infty,$ the Riesz norm $\rho(x)=(\int_0^1 |x(t)|^p dt)^{1/p}$ results in, in view of Lemma~\ref{UnboundedRiesz}, $\rho_u(x)=(\int_0^1 |x(t)\wedge u(t)|^p dt)^{1/p}$ where $x\in C[0,1].$
\end{exam}

Riesz pseudonorms obtained via Lemma~\ref{UnboundedRiesz} can be utilized with the induced pseudo-metric to obtain various classes of Lipschitzian operators which can be useful in the study of semigroups of operators. For a future reference, we record an example of a class of operators satisfying a uniform Lipschitz condition with respect to pseudo-metrics. 
\begin{exam}
An operator $T\colon X\to Y$ between two pseudonormed vector lattices $(X,\rho)$ and $(Y,\rho')$ is said to be \textit{non-expansive} if \[\varrho'(Tx,Ty)\leq \varrho(x,y) \] for all $x,y\in X$, see the comments above Proposition~\ref{pseudoMetricProp}. Consider $X=C[0,1],$ the vector lattice of continuous real valued functions on $[0,1]$ with the pointwise ordering. Let $u\in C[0,1]$ be positive. For each $1\leq p < \infty,$ the Riesz norm $\rho(x)=(\int_0^1 |x(t)|^p dt)^{1/p}$ results in $\rho_u(x)=(\int_0^1 |x(t)\wedge u(t)|^p dt)^{1/p}$ where $x\in C[0,1].$ An operator $T\colon C[0,1]\to C[0,1]$ is non-expansive with respect to the pseudo-metric $\varrho_u$ induced by $\rho_u$ if and only if \[ \int_{0}^{1} |(|Tx-Ty|\wedge u)(t)|^p dt\leq \int_{0}^{1} |(|x-y|\wedge u)(t)|^p dt \] for every $x,y\in C[0,1].$ 
\end{exam}

\begin{thm}
	Suppose that $\rho$ is a Riesz seminorm on a vector lattice $X$. Let $u\in X^+$ be arbitrary, and put $\rho_u(x)=\rho(|x|\wedge u)$ for $x\in X$.
	\begin{itemize}
		\item[\em i.] {If $\rho_u$ is such that $U_{\rho_u}$ is order bounded in $X$ then $U_{\rho}$ is order bounded in $X$. } 
		\item[\em ii.] {If $\rho_u$ is a bounded pseudonorm then $\rho$ is a bounded seminorm. }
		\item[\em iii.] {If $\rho_u$ is an almost bounded pseudonorm then $\rho$ is an almost bounded seminorm. }
		\item[\em iv.] {If $X$ is a normed lattice and $\rho$ is a coarsely almost bounded seminorm then $\rho_u$ is a coarsely almost bounded pseudonorm. }
		\item[\em v.] {If $\rho$ is an order continuous seminorm then $\rho_u$ is an order continuous pseudonorm. }
		\item[\em vi.] {If $\rho$ has the ideal (band) property then $\rho_u$ has the ideal (respectively, band) property. }
		
	\end{itemize}
\end{thm}

\begin{proof}
$(i).$	Let $x\in U_{\rho}$ so that $\rho(x)\leq 1$ holds. Since $\rho$ is a Riesz seminorm, $\rho(|x|)=\rho(x)\leq 1$ holds. It follows that $\rho(|x|\wedge u)\leq \rho(|x|)\leq 1$. Hence, $U_{\rho}\subseteq U_{\rho_u}$ for all $u\in X^+.$ Since the set $U_{\rho_u}\subseteq X$ is order bounded, the set $U_{\rho}$ is order bounded. 

$(ii).$ Let $B$ be a nonempty subset of $X$ which is pseudo-bounded with respect to the Riesz seminorm $\rho$. Since the set $\rho(B)$ is bounded, the set $\rho_u(B)$ is also bounded. As $\rho_u$ is a bounded pseudonorm, the set $B$ is order bounded. Hence, $\rho$ is a bounded seminorm.

$(iii).$ Let $B$ be almost pseudo-bounded with respect to $\rho$. By Lemma~\ref{UnboundedRiesz}, $\rho_u$ is a Riesz pseudonorm. For every $\epsilon>0$ there exists an order interval $[x,y]$ such that $B\subseteq [x,y]+\epsilon U_{\rho}.$ Hence, it follows from $B\subseteq [x,y]+\epsilon U_{\rho_u}$ that $B$ is almost pseudo-bounded with respect to $\rho_u$. As $\rho_u$ is an almost bounded pseudonorm, $B$ is order bounded in $X$. Hence, $\rho$ is an almost bounded seminorm.

$(iv).$ Let $B$ be almost order bounded in $X$. Since $\rho$ is a coarsely almost bounded seminorm, the set $B$ is almost pseudo-bounded with respect to $\rho$. For every $\epsilon>0$ there exists an order interval $[x,y]$ such that $B\subseteq [x,y]+\epsilon U_{\rho}.$ Hence, it follows from $B\subseteq [x,y]+\epsilon U_{\rho_u}$ that $B$ is almost pseudo-bounded with respect to $\rho_u$. Hence, $\rho_u$ is a coarsely almost bounded pseudonorm on $X$.

$(v).$ 	Let $x_{\alpha}$ be a net in $X$ satisfying $x_{\alpha}\xrightarrow{o} x$ for some $x\in X$. By the order continuity of $\rho$ we have $\rho(|x_{\alpha}-x|)\xrightarrow{}0$. As $\rho$ is a Riesz seminorm, $\rho_u(|x_{\alpha}-x|)=\rho(|x_{\alpha}-x|\wedge u)\leq \rho(|x_{\alpha}-x|)\xrightarrow{}0.$

$(vi).$ If the order ideal generated by $U_{\rho}$ is equal to $X$ then so is the order ideal general by $U_{\rho_u}$ for every $u\in X^+$.
\end{proof}

\begin{thm}
	Suppose that $\rho$ is a Riesz seminorm on a vector lattice $X$. Let $u\in X^+$, $U\subseteq X^+,$ $U\neq \emptyset$ and put $\rho_u(x)=\rho(|x|\wedge u)$ together with $\rho_U(x)=\sup_{u\in U} \rho_u(x),$ see~\cite[Section 3]{EV2017}, for $x\in X$.
	\begin{itemize}
		\item[\em i.] {If $\rho_u$ is a bounded pseudonorm for some $u\in U$ then $\rho_U$ is a bounded pseudonorm.}
		\item[\em ii.] {If $\rho_u$ is an almost bounded pseudonorm for some $u\in U$ then $\rho_U$ is an almost bounded pseudonorm. }
		\item[\em iii.] {If $X$ is a normed lattice and $\rho_U$ is a coarsely almost bounded pseudonorm then $\rho_u$ is a coarsely almost bounded pseudonorm for each $u\in U$. }
	\end{itemize}
\end{thm}

\begin{proof}
	$(i)$. Let $B$ be a nonempty subset of $X$ which is pseudo-bounded with respect to $\rho_U$. It follows from $\sup_{x\in B}\sup_{u\in U} \rho_u(x)<\infty$ that $B$ is pseudo-bounded with respect to $\rho_u$. As $\rho_u$ is a bounded pseudonorm, the set $B$ is order bounded.
	
	$(ii)$. Let $B$ be almost pseudo-bounded with respect to $\rho_U$. By Lemma~\ref{UnboundedRiesz}, $\rho_u$ is a Riesz pseudonorm for each $u\in U$. It follows that $\rho_{U}$ is a Riesz pseudonorm. For every $\epsilon>0$ there exists an order interval $[x,y]$ such that $B\subseteq [x,y]+\epsilon U_{\rho_{U}}.$ Consider the element $u\in U$ for which $\rho_u$ is an almost bounded pseudonorm. It follows from $B\subseteq [x,y]+\epsilon U_{\rho_u}$ that $B$ is almost pseudo-bounded with respect to $\rho_u$. As $\rho_u$ is an almost bounded pseudonorm, $B$ is order bounded in $X$. Hence, $\rho_{U}$ is an almost bounded seminorm.
	
	$(iii)$. Let $B$ be almost order bounded in $X$. Since $\rho_{U}$ is a coarsely almost bounded pseudonorm, $B$ is almost pseudo-bounded with respect to $\rho_{U}$. For every $\epsilon>0$ there exists an order interval $[x,y]$ such that $B\subseteq [x,y]+\epsilon U_{\rho_{U}}.$ It follows from $B\subseteq [x,y]+\epsilon U_{\rho_u}$ that $B$ is almost pseudo-bounded with respect to $\rho_u$.
\end{proof}

\begin{prop}
	Let $f_1,f_2$ be order bounded functionals on a normed lattice $X$ such that $f_1\leq f_2$. Let $\rho_i(x)=|f_i(|x|)|$ for $x\in X$ and $i=1,2.$
	\begin{itemize}
		\item[\em i.] {If $U_{\rho_1}$ is order bounded in $X$ then $U_{\rho_2}$ is order bounded. }
		\item[\em ii.] {If $\rho_1$ is a bounded seminorm on $X$ then $\rho_2$ is also a bounded seminorm. }
		\item[\em iii.] {If $\rho_2$ is a coarsely almost bounded seminorm then $\rho_1$ is also a coarsely almost bounded seminorm. }
	\end{itemize}
\end{prop}
\begin{proof}
$(i)$. It follows from $f_1\leq f_2$ that $U_{\rho_2}\subseteq U_{\rho_1}.$ Hence, if $U_{\rho_1}$ is order bounded then so is $U_{\rho_2}.$

$(ii)$. Suppose that $B\subseteq X$ is pseudo-bounded with respect to $\rho_2$. Since $\rho_2(B)$ is bounded and $f_1\leq f_2$ the set $\rho_1(B)$ is also bounded. As $\rho_1$ is a bounded, $B$ is order bounded in $X$. This shows that $\rho_2$ is a bounded seminorm.

$(iii)$ Let $B$ be an almost order bounded subset of $X$. As $\rho_2$ is a coarsely almost bounded, for every $\epsilon>0$ there exists an order interval $[x,y]$ in $X$ such that $B\subseteq [x,y]+\epsilon U_{\rho_2}.$ It follows from $B\subseteq [x,y]+\epsilon U_{\rho_1}$ that $B$ is almost pseudo-bounded with respect to $\rho_1.$ Hence, $\rho_1$ is a coarsely almost bounded seminorm. 
\end{proof}

\begin{lem}\label{LemmaMonotoneConv}
	Let $\rho$ be a Riesz pseudonorm on a vector lattice $X$ such that the generating cone $X^+$ is closed with respect to locally solid topology induced by $\rho$. Any monotone and pseudonorm convergent net in $X$ order converges to its pseudonorm limit in $X$.
\end{lem}
\begin{proof}
	Let $x_{\alpha}$ be a net such that $x_{\alpha}\uparrow$ and $x_{\alpha}\xrightarrow{\rho} x$. The generating cone $X^+$ of $X$ is closed with respect to the locally solid topology on $X$ induced by $\rho$. Fix an arbitrary index $\alpha$. Then $x_{\beta}-x_{\alpha}\in X^+$ whenever $\beta\geq \alpha$. By taking limit of $x_{\beta}-x_{\alpha}$ over $\beta$ with respect to pseudonorm $\rho,$ we conclude that $x-x_{\alpha}\in X^+,$ and hence, $x\geq x_{\alpha}$ for any $\alpha$. Since $\alpha$ is arbitrary, $x$ is an upper bound of $x_{\alpha}.$ If $y\geq x_{\alpha}$ for all $\alpha$ then $y-x_{\alpha}\xrightarrow{\rho} y-x\in X^+.$ Hence, $y\geq x$ implies $x_{\alpha}\uparrow x$.
\end{proof}

\begin{coro}
	Let $\rho$ be a Riesz pseudonorm on a vector lattice $X$ such that the generating cone $X^+$ is closed with respect to locally solid topology induced by $\rho$. Let $x_{\alpha}, y_{\alpha}$ be two increasing nets in $X$ such that $x_{\alpha}\xrightarrow{\rho}x$ and $y_{\alpha}\xrightarrow{\rho}y$ and $x_{\alpha}\perp y_{\alpha}$ for all $\alpha$. Then $x\perp y.$ 
\end{coro}
\begin{proof}
	By Lemma~\ref{LemmaMonotoneConv}, both $x_{\alpha}$ and $y_{\alpha}$ order converge to their pseudonorm limits, i.e., $x_{\alpha}\xrightarrow{o}x$ and $y_{\alpha}\xrightarrow{o}y$. It follows from $x_{\alpha}\perp y_{\alpha}$ for all $\alpha$ and~\cite[Chapter III.7]{V1967} that $x\perp y$.
\end{proof}

\begin{prop}\label{Prop11006}
	Let $\rho$ be a Riesz pseudonorm on a vector lattice $X$. The pseudonorm $\rho$ is order continuous if and only if $x_{\alpha}\downarrow 0$ implies $\rho(x_{\alpha}) \xrightarrow{} 0$.
\end{prop}
\begin{proof}
	Forward implication is clear. Let $x_{\alpha}\xrightarrow{o} 0$ so that there exists $z_{\beta}\downarrow 0,$ with a possibly different index set, such that for any $\beta$ there exists $\alpha_{\beta}$ satisfying $|x_{\alpha}|\leq z_{\beta}$ for all $\alpha\geq \alpha_{\beta}$. Since $\rho$ is a Riesz pseudonorm, $\rho(x_{\alpha})\leq \rho(z_{\beta})$ for all $\alpha\geq \alpha_{\beta}$. By the assumption, $\rho(z_{\beta})\downarrow 0$. Hence, $\rho(x_{\alpha})\xrightarrow{}0.$ This shows that the Riesz pseudonorm $\rho$ is order continuous.
\end{proof}

\begin{thm}
	Let $\rho$ be a Riesz pseudonorm on a vector lattice $X$ such that the generating cone $X^+$ is closed and $X$ is complete with respect to the locally solid topology on $X$ induced by $\rho$. Following statements are equivalent: 
		\begin{itemize}
		\item[\em i.] {The pseudonorm $\rho$ is order continuous;} 
		\item[\em ii.] {if $0\leq x_{\alpha}\uparrow x$ holds in $X$ then $x_{\alpha}$ is a Cauchy net in the locally solid topology on $X$ induced by $\rho$;}
		\item[\em iiii.] {$x_{\alpha}\downarrow 0$ in $X$ implies $\rho(x_{\alpha})\xrightarrow{} 0$. }
	\end{itemize}
\end{thm}
\begin{proof}
$(i)\Rightarrow (ii)$. Let $0\leq x_{\alpha}\uparrow \leq x$ in the vector lattice $X$. There exists a net $y_{\beta}$ in $X$ such that $(y_{\beta}-x_{\alpha})_{\alpha,\beta}\downarrow 0.$ Hence $\rho(y_{\beta}-x_{\alpha})\to 0$ and the net $x_{\alpha}$ is Cauchy with respect to the locally solid topology on $X$ induced by $\rho$. 

$(ii)\Rightarrow (iii)$. Let $x_{\alpha}\downarrow 0$ in $X$. Fix an arbitrary $\alpha_0.$ It follows that $x_{\alpha}\leq x_{\alpha_0}$ for $\alpha\geq \alpha_0$ and that $0\leq (x_{\alpha_0}-x_{\alpha})_{\alpha\geq \alpha_0}\uparrow \leq x_{\alpha_0}.$ The net $(x_{\alpha_0}-x_{\alpha})_{\alpha\geq \alpha_0}$ is Cauchy with respect to the locally solid topology induced by $\rho$. As $X$ is complete with respect to this topology, there exists an $x\in X$ satisfying $\rho(x-x_{\alpha})\to 0.$ It follows from Lemma~\ref{LemmaMonotoneConv} that $x=0.$ Hence, $x_{\alpha}$ is pseudonorm converges to $0$ and $\rho(x_{\alpha})\to 0.$

$(iii)\Rightarrow (i)$. Follows from Proposition~\ref{Prop11006}.
\end{proof}

\begin{prop}
	Let $\rho$ be a Riesz pseudonorm on a vector lattice $X$. If $\rho$ is order continuous and $X$ is complete with respect to the locally solid topology induced by $\rho$ then $X$ is order complete.
\end{prop}
\begin{proof}
	Let $0\leq x_{\alpha}\uparrow u$ It follows that $x_{\alpha}$ is a Cauchy net with respect to the locally solid topology induced by $\rho$. Since $X$ is complete, there is an $x\in X^+$ such that $x_{\alpha}\xrightarrow{\rho}x.$ Hence, $x_{\alpha}\uparrow x,$ and $X$ is order complete.
\end{proof}

\section{Compact Operators between Pseudonormed Vector Lattices}\label{PseudoCompact}

We continue to follow the convention given in Section~\ref{PseudoBasics} that when the pseudonorm in discussion is clear from the context, we write $X$ for the pseudonormed vector lattice $(X,\rho)$.

\begin{definition}\label{DefnPsCompact}
	An operator $T\colon X\to Y$ between two pseudonormed vector lattices $X$ and $Y$ is said to be \textit{pseudonorm compact} if for any pseudo-bounded net $x_{\alpha}$ in $X,$ there is a subnet $x_{\alpha_\beta}$ and $y\in Y$ such that the net $Tx_{\alpha_\beta}$ is pseudonorm convergent to $y\in Y$.
\end{definition}

It follows from Definition~\ref{DefnPsCompact} that an operator $T\colon X\to Y$ between two pseudonormed vector lattices $X$ and $Y$ is pseudonorm compact if and only if $T$ maps pseudo-bounded subsets of $X$ into conditionally pseudonorm compact subsets of $Y$, see Section~\ref{PseudoBasics}.

An operator $T\colon X\to Y$ is said to be \textit{sequentially pseudonorm compact} if for any pseudo-bounded sequence $x_n$ in $X$ there is a subsequence $x_{n_k}$ such that the sequence $Tx_{n_k}$ is pseudonorm convergent in $Y$.

\begin{exam} In this example we show that pseudonorm compact operators canonically generalize the class of compact operators. Let $X$ and $Y$ be normed vector lattices. Since a Riesz norm is in particular a Riesz pseudonorm, an operator $T\colon X\to Y$ is (sequentially) pseudonorm compact with respect to norms if and only if $T\colon X\to Y$ is compact. It is instructive to recall one of the most well-understood example of compact operators. Consider the Banach lattice $C[a,b]$ of continuous real valued functions on the closed real interval $[a,b]$ with the uniform norm. A degenerate kernel on $[a,b]$ is a a function of the form $k(x,y)=\sum_{i=1}^n\beta_i(x)\gamma_i(y)$ in which each $\beta_i$ is continuous on $[a,b]$ and $\gamma_i$ is absolutely integrable on $[a,b]$ for $i=1,2,\ldots,n.$ The induced operator $K\colon C[a,b]\to C[a,b]$ with $Kf(x)=\sum_{i=1}^{n}\beta_i(x)\int_{a}^{b}\gamma_i(y)f(y)dy$ for $f\in C[a,b] $ is compact.
\end{exam}

Let $X$ be a Banach lattice and $Y$ be a Dedekind complete normed lattice. We recall that an operator $T\colon X\to Y$ is called \textit{$AM$-compact} if and only if $T[-x,x]$ is relatively compact for every $x\in X^+$, see the recent articles~\cite{AEEM2018, Pli2011} and the references therein. In this case, if $T\colon X\to Y$ is a compact operator $X$ and $Y$ then $T$ is $AM$-compact.
\begin{thm}\label{PAMPseudo}
	Let $\rho$ be a bounded pseudonorm on a Banach lattice $X$. Let $(Y,\norm{\cdot}_Y)$ be a Dedekind complete normed lattice. If an operator $T\colon X\to Y$ is $AM$-compact then $T$ is pseudonorm compact with respect to $\rho$ and $\norm{\cdot}_Y$. 
\end{thm}
\begin{proof}
	Let $x_{\alpha}$ be a pseudo-bounded net with respect to $\rho$. Since $\rho$ is a bounded pseudonorm, the net $x_{\alpha}$ is order bounded in $X$. Since $T$ is $AM$-compact, there exists a subnet $x_{\alpha_{\beta}}$ and $y\in Y$ such that $Tx_{\alpha_{\beta}}\xrightarrow{\norm{\cdot}_Y} y.$ This shows that $T$ is pseudonorm compact with respect to $\rho$ and $\norm{\cdot}_Y$. 
\end{proof}
\begin{exam}\label{ExamBoundedPseu}
	Suppose that $X$ is an $AM$-space with a strong norm unit, see Example~\ref{ExampleAMFirst}. Let $f\in X',$ the topological dual of $X,$ and $\rho_f(x)=|f(x)|$. If $B\subseteq X$ is pseudo-bounded with respect to $\rho_f$ then it is norm bounded in $X$. Hence, the subset $B$ is order bounded in $X$. It follows that $\rho_f$ is a bounded pseudonorm on $X$. In the settings of Theorem~\ref{PAMPseudo}, if an operator $T\colon X\to Y$ is $AM$-compact then $T$ is pseudonorm compact with respect to $\rho_f$ and $\norm{\cdot}_Y$.
\end{exam}

We recall from Section~\ref{PseudoBasics}, also see~\cite{EV2017}, that $\rho_u(x)=\rho(|x|\wedge u)$ for $x\in X$ and $u\in X^+$. 
\begin{prop}
	Suppose that $T\colon (X,\rho)\to (Y,\rho')$ is a pseudonorm compact operator between two seminormed vector lattices $(X,\rho)$ and $(Y,\rho')$. Let $u\in X^+$ and $v\in Y^+$ be arbitrary. If $\rho_u$ is a bounded pseudonorm on $X$ then $T$ is pseudonorm compact with respect to pseudonorms $\rho_u$ and $\rho'_v.$
\end{prop}
\begin{proof}
	Let $(x_{\alpha})\subseteq X$ be pseudo-bounded with respect to $\rho_u.$ Since $\rho_u$ is a bounded pseudonorm on $X,$ the net $x_{\alpha}$ is order bounded in $X$. Thus, $x_{\alpha}$ is pseudo-bounded with respect to $\rho$. As $T\colon (X,\rho)\to (Y,\rho')$ is pseudonorm compact, there exists a subnet $x_{\alpha_{\beta}}$ and $y\in Y$ such that $Tx_{\alpha_{\beta}}\xrightarrow{\rho'}y$. We note that $\rho'$ is a Riesz seminorm. These imply that $Tx_{\alpha_{\beta}}\xrightarrow{\rho'_v}y$ for any $v\in Y^+$. Hence the operator $T$ is pseudonorm compact with respect to pseudonorms $\rho_u$ and $\rho'_v$.
\end{proof}

\begin{coro}
	Suppose that $X$ is an order continuous Banach lattice and $u,v\in X^+$. Let $\rho$ be a Riesz seminorm on $X$ such that $\rho$ is coarser than the norm and that $\rho_u$ is bounded. If an operator $T\colon X\to X$ is $AM$-compact then $T$ is pseudonorm compact with respect to pseudonorms $\rho_u$ and $\rho_v.$
\end{coro}
\begin{proof}
	Since $T\colon (X,\norm{\cdot})\to (X,\norm{\cdot}) $ is $AM$-compact and $\rho_u$ is a bounded pseudonorm on $X,$ it follows from Theorem~\ref{PAMPseudo} that $T\colon (X,\rho_u)\to (X,\norm{\cdot})$ is pseudonorm compact. As $\rho$ is a Riesz seminorm on $X$ which is coarser than the norm, $\rho_v$ is a Riesz pseudonorm on $X$ which is coarser than the norm. It follows that $T\colon (X,\rho_u)\to (X,\rho_v)$ is pseudonorm compact. 
\end{proof}

We recall that an operator $T\colon X\to Y$ between two vector lattices $X$ and $Y$ is said to be \textit{order compact}, see~\cite{EGGS2019}, if for every order bounded net $x_{\alpha}$ in $X$ there exists a subnet $x_{\alpha_{\beta}}$ and $y\in Y$ such that $Tx_{\alpha_{\beta}}\xrightarrow{o} y$ in $Y$.

\begin{prop}\label{PPseudoOrder}
	Suppose that $(X,\rho)$ and $(Y,\rho')$ are pseudonormed vector lattices such that $\rho'$ has the subnet property. If an operator $T\colon X\to Y$ is pseudonorm compact then $T$ is order compact.
\end{prop}
\begin{proof}
	Let $x_{\alpha}$ be an order bounded net in $X.$ The net $x_{\alpha}$ is pseudo-bounded with respect to $\rho$. Hence, there exists a subnet $x_{\alpha_{\beta}}$ and $y\in Y$ such that $Tx_{\alpha_{\beta}}\xrightarrow{\rho'} y.$ As $\rho'$ has the subnet property, $Tx_{\alpha_{\beta_{\gamma}}}\xrightarrow{o} y$ for a further subnet $x_{\alpha_{\beta_{\gamma}}}$ of $x_{\alpha_{\beta}}$. This shows that the operator $T\colon X\to Y$ is order compact. 
\end{proof}

In views of Proposition~\ref{PPseudoOrder}, Example~\ref{AnExample} and~\cite[Theorem VII.2.1]{V1967}, we have the following result. 

\begin{thm}\label{PPseudoOrderSeq}
	Suppose that $(X,\rho)$ and $(Y,\rho')$ are pseudonormed vector lattices such that $\rho'$ has the subsequence property. If an operator $T\colon X\to Y$ is sequentially pseudonorm compact then $T$ is sequentially order compact.
\end{thm}
\begin{proof}
	The proof is similar to that of Proposition~\ref{PPseudoOrder}. Let $x_n$ be an order bounded sequence in $X$. Hence, the sequence $x_n$ is pseudo-bounded with respect to $\rho$. As the operator $T\colon X\to Y$ is pseudonorm compact, there exists a subsequence $x_{n_k}$ and $y\in F$ such that $Tx_{n_k}\xrightarrow{\rho'} y$ in $Y$. Since $\rho'$ has the subsequence property, there exists a further subsequence $x_{n_{k_m}}$ such that $Tx_{n_k}\xrightarrow{o} y$ in $Y$. Hence, the operator $T$ is sequentially order compact.
\end{proof}

\begin{thm}\label{Porderpseudo}
	Suppose that $(X,\rho)$ and $(Y,\rho')$ are pseudonormed vector lattices such that $\rho$ is bounded and $\rho'$ is order continuous. If an operator $T\colon X\to Y$ is order compact then $T$ is pseudonorm compact.
\end{thm}
\begin{proof}
	Let $x_{\alpha}$ be a pseudo-bounded net in $X$. As $\rho$ is bounded, the net $x_{\alpha}$ is order bounded in $X$. Since $T$ is order compact, there exists a subnet $x_{\alpha_{\beta}}$ and $y\in Y$ such that $Tx_{\alpha_{\beta}}\xrightarrow{o}y.$ As $\rho'$ is order continuous, we have $Tx_{\alpha_{\beta}}\xrightarrow{\rho'}y.$ Hence, $T$ is pseudonorm compact with respect to pseudonorms $\rho$ and $\rho'$.
\end{proof}
\begin{exam}
	Suppose that $X$ is an $AM$-space with a strong norm unit. Let $f\in X',$ the topological dual of $X,$ and $\rho_f(x)=|f(x)|$. Then $\rho_f$ is a bounded pseudonorm, see Example~\ref{ExamBoundedPseu}. Further, let $(Y,\norm{\cdot})$ be an order continuous Banach lattice. By Theorem~\ref{Porderpseudo}, if an operator $T\colon X\to Y$ is order compact then $T$ is pseudonorm compact with respect to $\rho_f$ and $\norm{\cdot}$.
\end{exam}

\begin{coro}
 Suppose that $X$ is an $AM$-space with a strong norm unit and that $Y$ is an order continuous Banach lattice. If an operator $T\colon X\to Y$ is order compact then $T$ is compact.
\end{coro}
\begin{proof}
	By Theorem~\ref{Porderpseudo}, $T\colon X\to Y$ is pseudonorm compact with respect to norms. This is equivalent to saying that $T$ is a compact operator.
\end{proof}

An operator $T\colon X\to Y$ is said to be \textit{pseudo-bounded} if $T$ maps pseudo-bounded sets into pseudo-bounded sets. It is a well-known fact that boundedly bounded operators, those operators between topological vector spaces mapping topologically bounded sets into topologically bounded sets, has diverse field of applications. See Proposition~\ref{bbOperator} and Remark~\ref{RemContinuity} for some applications in the present settings.

\begin{lem}\label{IdealProperty1} Suppose that $L,T,R\colon X\to X$ are operators on a pseudonormed vector lattice $(X,\rho)$. 
			\begin{itemize}
			\item[\em i.] {If $T$ is pseudonorm compact and $L$ is pseudonorm continuous then $LT$ is pseudonorm compact.} 
			\item[\em ii.] {If $T$ is pseudonorm compact and $R$ is pseudo-bounded then $TR$ is pseudonorm compact. }
		\end{itemize}
\end{lem}
\begin{proof}
		$(i)$. Let $x_{\alpha}$ be a pseudo-bounded net in $X$. There exists a subnet $x_{\alpha_{\beta}}$ and $x\in X$ such that $Tx_{\alpha_{\beta}}\xrightarrow{\rho} x$. Hence, $LT(x_{\alpha_{\beta}})\xrightarrow{\rho}L(x)$. Hence, $LT$ is pseudonorm compact.
		
		$(ii)$. Let $x_{\alpha}$ be a pseudo-bounded net. Since $R$ is pseudo-bounded, $Rx_{\alpha}$ is pseudo-bounded. As $T$ is pseudonorm compact, there exists a subnet $x_{\alpha_{\beta}}$ and $x\in X$ such that $TR(x_{\alpha_{\beta}})\xrightarrow{\rho}x$.
\end{proof}

\begin{thm}\label{NewPseudonorms1}
	 Suppose that $(X,\rho)$ is a seminormed vector lattice. Let $B\subseteq X$ be a pseudo-bounded subset of $(X,\rho)$, i.e., $\rho(B)<\infty$. The function $\rho_B(T)=\sup_{x\in B}\rho(Tx)$ is a pseudonorm on the linear space of all pseudo-bounded operators $T\colon X\to X$.
\end{thm}
\begin{proof}
	We have \[\lim_{\theta\to 0}\rho_B(\theta T)\leq \lim_{\theta\to 0}|\theta|\sup_{x\in B}\rho(T(x))=0\] and \[\rho_B(T+S)=\sup_{x\in B}\rho(Tx+Sx)\leq \sup_{x\in B}\rho(Tx)+\sup_{x\in B}\rho(Sx)=\rho_B(T)+\rho_B(S)\] whenever $T$ and $S$ are pseudo-bounded operators.
\end{proof}
\begin{prop}\label{PatomicKB}
	Let $\rho$ be a bounded pseudonorm on a vector lattice $X$. Let $\rho'$ be an order continuous pseudonorm on an atomic $KB$-space $(Y,\norm{\cdot})$. Every order bounded operator $T\colon X\to Y$ is pseudonorm compact with respect to $\rho$ and $\rho'$.
\end{prop}
\begin{proof}
	Let $x_{\alpha}$ be a pseudo-bounded net in $X$. Since the pseudonorm $\rho$ on $X$ is bounded, the net $x_{\alpha}$ is order bounded in $X$. Since $T$ is an order bounded operator, the net $Tx_{\alpha}$ is order bounded in the atomic $KB$-space $(Y,\norm{\cdot})$. It is a classical fact that every order bounded net in an atomic $KB$-space has an order convergent subnet. Hence, there exists a subnet $x_{\alpha_{\beta}}$ such that the net $Tx_{\alpha_{\beta}}$ is order convergent. Since $\rho'$ is order continuous, the subnet $Tx_{\alpha_{\beta}}$ is pseudonorm convergent with respect to $\rho'$. Hence, $T\colon X\to Y$ is pseudonorm compact with respect to $\rho$ and $\rho'$.
\end{proof}

\begin{exam}
	We cannot drop atomicity in Proposition~\ref{PatomicKB}. Indeed, consider the identity operator on $L_1[0,1]$. The sequence of Rademacher functions is order bounded and has no order convergent subsequence. Hence, the identity operator on $L_1[0,1]$ is not a compact; and hence, it is not pseudonorm compact with respect to canonical norm on $L_1[0,1]$.
\end{exam}
\begin{prop}
	Suppose that $(X,\rho)$ and $(Y,\rho')$ are pseudonormed vector lattices such that $\rho$ is bounded. Let $T\colon X\to Y$ be an order bounded finite rank operator. Then $T$ is pseudonorm compact with respect to $\rho$ and $\rho'$. In particular, $T=\sum_{i=1}^{n}f_i\otimes y_i$ is pseudonorm compact whenever $f_1,f_2,\ldots,f_n$ are order bounded functionals on $X,$ see~\cite[p.64]{AB2}, and $y_1,y_2,\ldots, y_n\in Y$. 
\end{prop}
\begin{proof}
	We may suppose that $T$ is given by $Tx=f(x)y$ for some order bounded functional $f$ on $X$ and $y\in Y.$ Let $x_{\alpha}$ be a pseudo-bounded net with respect to $\rho.$ Since $\rho$ is a bounded pseudonorm, the net $x_{\alpha}$ is order bounded in $X.$ There is a subnet $x_{\alpha_{\beta}}$ such that $f(x_{\alpha_{\beta}})\xrightarrow{}\lambda$ for some $\lambda\in \rr$ because the functional $f$ is order bounded. It follows from the definition of pseudonorms that $\rho'(T(x_{\alpha_{\beta}})-\lambda y)=\rho'((f(x_{\alpha_{\beta}})-\lambda)y)\xrightarrow{} 0$. Thus, $T$ is pseudonorm compact with respect to $\rho$ and $\rho'$.
\end{proof}

We recall that an operator $T\colon X\to Y$ from a normed space $X$ into a normed lattice $Y$ is called \textit{semicompact} if $T(B_X),$ the image of the closed unit ball of $X$ under $T,$ is almost order bounded in $Y,$ see~\cite{AB2,AEEM2018, MN91}.
\begin{definition}\label{DefnSemicompactVer}
	An operator $T\colon X\to Y$ between pseudonormed vector lattices is said to be pseudo-semicompact if the images of pseudo-bounded sets under $T$ are almost pseudo-bounded. An operator $T\colon X\to Y$ from a pseudonormed vector lattice $X$ into a Banach space $Y$ is called \textit{pseudo-$AM$-compact} if for any pseudo-bounded set $B$ in $X,$ the set $T(B)$ is relatively compact in $Y$. 
\end{definition}

Let $X$ and $Y$ be pseudonormed vector lattices. It follows that if $T\colon X\to Y$ is pseudo-semicompact then $T$ is pseudo-bounded, see the comment given above Lemma~\ref{IdealProperty1}. Furthermore, if $T,S\colon X\to X$ are such that $T$ is pseudo-semicompact and $S$ is pseudonorm compact then $ST$ is pseudonorm compact because almost pseudo-bounded sets are pseudo-bounded. 

\begin{prop} Suppose that $(X,\rho)$ is a seminormed vector lattice. Let $B\subseteq X$ be a pseudo-bounded subset of $(X,\rho)$, i.e., $\rho(B)<\infty$. The function $\rho_B(T)=\sup_{x\in B}\rho(Tx)$ is a pseudonorm on the linear space of all pseudo-semicompact operators $T\colon X\to X$.
\end{prop}
\begin{proof}
	We first note that almost pseudo-bounded subsets of $X$ are also pseudo-bounded. Hence, if an operator $T\colon X\to X$ is pseudo-semicompact, then it is pseudo-bounded. It follows from Theorem~\ref{NewPseudonorms1} that \[\rho_B(T)=\sup_{x\in B}\rho(Tx)\] is also a pseudonorm on the linear space of all pseudo-semicompact operators $T\colon X\to X$.
\end{proof}
\begin{prop}
	Let $X$ and $Y$ be two normed lattices. Suppose that $\rho$ is a pseudonorm on $X$ which is coarser than the norm, and that, $\rho'$ is an almost bounded pseudonorm on $Y$. If an operator $T\colon X\to Y$ is pseudo-semicompact with respect to $\rho$ and $\rho'$ then $T$ is semicompact. 
\end{prop}
\begin{proof}
	Let $B$ be a norm bounded subset of $X$. As $\rho$ is coarser than the norm of $X,$ the set $B$ is pseudo-bounded with respect to $\rho$. As $T\colon X\to Y$ is pseudo-semicompact, the set $T(B)$ is almost pseudo-bounded with respect to $\rho'$. Since $\rho'$ is an almost bounded pseudonorm on $Y,$ the set $T(B)$ is order bounded in $Y$. Every order bounded subset of $Y$ is almost order bounded. Hence, $T(B)$ is almost order bounded. This shows that $T\colon X\to Y$ is a semicompact operator.
\end{proof}
\begin{prop}
	Suppose that $T\colon (X,\rho)\to (Y,\rho')$ is a pseudo-semicompact operator between two pseudonormed vector lattices $(X,\rho)$ and $(Y,\rho')$. Let $u\in X^+$ and $v\in Y^+$ be arbitrary. If $\rho_u$ is a bounded pseudonorm on $X$ then $T$ is pseudo-semicompact with respect to $\rho_u$ and $\rho'_v.$
\end{prop}
\begin{proof}
	Let $B\subseteq X$ be a pseudo-bounded set with respect to $\rho_u$. Since $\rho_u$ is a bounded pseudonorm on $X$, the set $B$ is order bounded in $X$. Hence, $B$ is pseudo-bounded with respect to $\rho$. As $T\colon (X,\rho)\to (Y,\rho')$ is pseudo-semicompact, the set $T(B)$ is almost pseudo-bounded with respect to $\rho'$. It follows that the set $T(B)$ is almost pseudo-bounded with respect to $\rho'_v.$ Hence the operator $T$ is pseudo-semicompact with respect to $\rho_u$ and $\rho'_v$.
\end{proof}

\begin{prop}
	Suppose that $T\colon X\to Y$ is an interval preserving operator, see~\cite[page 94]{AB2}, between two Banach lattices $X$ and $Y$. Let $\rho$ be a semibounded pseudonorm on $X$. Then $T$ is pseudo-semicompact with respect to $\rho$ and $\norm{\cdot}_Y$.
\end{prop}
\begin{proof}
 Let $B\subseteq X$ be pseudo-bounded with respect to $\rho$. Since $\rho$ is a semibounded pseudonorm, the set $B$ is almost pseudo-bounded with respect to $\rho$. We note that as $T\colon X\to Y$ is interval preserving, $T$ is norm bounded. We denote by $\norm{T}$ the operator norm of $T$. Given $\epsilon>0$ there is an order interval $[x',y']\subseteq X$ such that $B\subseteq [x',y']+\epsilon' U_{\rho}$ where $\epsilon'=\epsilon/\norm{T}.$ We put $T([x',y'])=[x'',y'']$ for some $x'',y''\in Y$. It follows that $T(B)\subseteq [x'',y'']+\epsilon' T(U_{\rho})\subseteq [x'',y'']+ \epsilon U_{Y}. $ Hence, the set $T(B)$ is almost order bounded in $Y$. This means that the operator $T$ is pseudo-semicompact with respect to $\rho$ and $\norm{\cdot}_Y$.
\end{proof}

\begin{prop}
	Let $X$ and $Y$ be two normed lattices. Suppose that $\rho$ is a bounded pseudonorm on $X,$ and that, $\rho'$ is a coarsely almost bounded pseudonorm on $Y$. If an operator $T\colon X\to Y$ is semicompact then $T$ is pseudo-semicompact with respect to $\rho$ and $\rho'$. 
\end{prop}
\begin{proof}
	Let $B\subseteq X$ be a pseudo-bounded set with respect to pseudonorm $\rho$. As $\rho$ is a bounded pseudonorm, the set $B$ is order bounded in $X$. In particular, the set $B$ is norm bounded in $X$. Since $T\colon X\to Y$ is semicompact, the set $T(B)$ is almost order bounded in $Y$. As $\rho'$ is a coarsely almost bounded pseudonorm, the set $T(B)$ is almost pseudo-bounded with respect to $\rho'$. This shows that $T$ is pseudo-semicompact with respect to pseudonorms $\rho$ and $\rho'$. 
\end{proof}

\begin{prop}
	Suppose that $(X,\rho)$ and $(Y,\rho')$ are pseudonormed vector lattices. If an operator $T\colon X\to Y$ is pseudonorm compact then $T$ is pseudo-semicompact.
\end{prop}
\begin{proof}
	Let $B \subseteq X$ be a pseudo-bounded set with respect to $\rho$. As the operator $T\colon X\to Y$ is pseudonorm compact, the subset $T(B)$ is conditionally pseudonorm compact with respect to $\rho'$. We claim that $T(B)$ is almost pseudo-bounded with respect to $\rho'$. Suppose that there exists an $\epsilon>0$ such that whenever $y_1,y_2,\ldots,y_n$ belong to $T(B)$ there exists $y_{n+1}\in T(B)$ satisfying $\rho'(y_k,y_{n+1})\geq \epsilon$ for $k=1,2,\ldots,n.$ Then the no subsequence of $(y_k)_k$ in $T(B)$ pseudonorm converges in $Y$ with respect to $\rho'$ as $(y_n)_n$ has no Cauchy subsequence. Since $T(B)$ is conditionally pseudonorm compact, given $\epsilon>0$ there exists $y_1,y_2,\ldots,y_n$ in $T(B)$ such that $T(B)\subseteq \cup_{j=1}^n \{y\in Y\colon \rho'(y_j,y)<\epsilon \}.$ Hence, if $y\in T(B)$ then there exists some $j\in \{ 1,2,\ldots,n \}$ such that $y=y_j+y_j'$ where $\rho'(y_j')\leq \epsilon.$ It follows that if $y'=\inf \{ y_1,y_y,\ldots,y_n\}$ and $y''=\sup\{y_1,y_2,\ldots,y_n\}$ then $T(B)\subseteq [y',y'']+\epsilon U_{\rho'}. $ Hence, $T(B)$ is almost pseudo-bounded with respect to $\rho'$. This shows that the operator $T\colon X\to Y$ is pseudo-semicompact. 
\end{proof}

\begin{thm}
	Suppose that $X$ is an $AM$-space with a strong norm unit and that $\rho$ is a Riesz pseudonorm on $X$. Suppose further that $Y$ is a normed vector lattice and $\rho'$ is a seminorm on $Y$ which is finer than the norm. If an operator $T\colon X\to Y$ is pseudo-semicompact with respect to $\rho$ and $\rho'$ then $T\colon X\to Y$ is semicompact.
\end{thm}
\begin{proof}
	Consider the closed unit ball $B_X$ of $X$. As $X$ is an $AM$-space with a strong norm unit, there exists $e\in X^+$ such that $x\leq e$ for all $x\in B_X.$ Hence, the set $B_X$ is order bounded in $X$. In particular, $B_X$ is pseudo-bounded with respect to Riesz pseudonorm $\rho$. Since the operator $T$ is pseudo-semicompact with respect to $\rho$ and $\rho'$, the set $T(B_X)$ is almost pseudo-bounded in $Y$ with respect to $\rho'$. The seminorm $\rho'$ on $Y$ is finer than the norm. It follows that \[U_{\rho'}=\{y\in Y\colon \rho'(y)\leq 1\}\subseteq \{y\in Y\colon \norm{y}_Y\leq 1 \}=B_Y. \] In view of Lemma~\ref{Ppseudoalmostbounded}, given $\epsilon>0$ there is some $u_{\epsilon} \in Y^+$ such that $\rho(|Tx|-u_{\epsilon}\wedge |Tx|)\leq \epsilon$ for every $x\in B_X$. Hence, we have \[ \norm{(|Tx|-u_{\epsilon})^+} \leq \rho(|Tx|-u_{\epsilon}\wedge |Tx|)\leq \epsilon\] for every $x\in B_X$. This shows that the operator $T\colon X\to Y$ is semicompact.
\end{proof}

\begin{thm}
	Suppose that $(X,\rho)$ and $(Y,\rho')$ are seminormed vector lattices. Let $T\colon X\to Y$ be a positive pseudo-semicompact operator with respect to $\rho$ and $\rho'$. If $S\colon X\to Y$ is such that $0\leq S\leq T$ then $S$ is pseudo-semicompact with respect to $\rho$ and $\rho'$.
\end{thm}
\begin{proof}
	Let $B\subseteq X$ satisfy $\rho(B)< \infty$. It follows from the definition of Riesz seminorm that the set $|B|=\{|x|\colon x\in B\}$ satisfies $\rho(|B|)< \infty$. Since the operator $T\colon X\to Y$ is pseudo-semicompact, it follows that $T(|B|)$ is almost pseudo-bounded in $Y$ with respect to $\rho'$. By Lemma~\ref{Ppseudoalmostbounded}, given $\epsilon>0$ there is $u_{\epsilon}\in Y^+$ such that $\rho'((|Tx|-u_{\epsilon})^+)\leq \epsilon$ for all $x\in B$. Hence, $S|x|\leq T|x|$ implies that $(S|x|-u_{\epsilon})^+\leq (T|x|-u_{\epsilon})^+.$ Hence, $\rho'((S|x|-u_{\epsilon})^+)\leq \epsilon$ for every $x\in B$. It follows from $(|Sx|-u_{\epsilon})^+\leq (S|x|-u_{\epsilon})^+$ that \[ \rho'(|Sx|-u_{\epsilon})^+\leq \rho(S|x|-u_{\epsilon})^+\leq \epsilon \] for all $x\in B$. This shows that the operator $S\colon X\to Y$ is pseudo-semicompact with respect to $\rho$ and $\rho'$.
\end{proof}

\begin{prop}
	Suppose that $(X,\norm{\cdot}_X)$ and $(Y,\norm{\cdot}_Y)$ are Banach lattices. Suppose further that $\rho$ is a pseudonorm on $X$ which is coarser than $\norm{\cdot}_X$. If an operator $T\colon (X,\rho)\to (Y,\norm{\cdot}_Y)$ is pseudo-$AM$-compact, see Definition~\ref{DefnSemicompactVer}, then $T\colon (X,\norm{\cdot}_X)\to (Y,\norm{\cdot}_Y)$ is semicompact.
\end{prop}
\begin{proof}
	Let $B$ be a norm bounded subset of $X$. Since the pseudonorm $\rho$ is coarser than the norm $\norm{\cdot}_X$, the set $B$ is pseudo-bounded with respect to $\rho$. As $T\colon (X,\rho)\to (Y,\norm{\cdot}_Y)$ is pseudo-$AM$-compact, the set $T(B)$ is relatively compact in $Y$. Relatively compact subsets of $Y$ are order bounded, and hence, almost order bounded in $Y$. Hence the operator $T\colon (X,\norm{\cdot}_X)\to (Y,\norm{\cdot}_Y)$ is semicompact. 
\end{proof}

\begin{prop}
	Suppose that $(X,\rho)$ is a pseudonormed vector lattice with a bounded pseudonorm $\rho$. Let $(Y,\norm{\cdot}_Y)$ be a Banach lattice. If an operator $T\colon X\to Y$ is pseudonorm compact then $T$ is pseudo-$AM$-compact with respect to $\rho$ and $\norm{}_Y$. 
\end{prop}
\begin{proof}
	 Let $B\subseteq X$ be a pseudo-bounded set with respect to $\rho$. Since $\rho$ is a bounded pseudonorm, the set $B$ is order bounded in $X$. As the operator $T\colon X\to Y$ is pseudonorm compact, the set $T(B)$ is conditionally compact in the Banach lattice $(Y,\norm{\cdot}_Y).$ Hence, the operator $T\colon X\to Y$ is pseudo-$AM$-compact with respect to $\rho$ and $\norm{}_Y$.
\end{proof}

\begin{prop}
	Suppose that $(X,\rho)$ is a pseudonormed vector lattice. Let $(Y,\norm{\cdot}_Y)$ be a Banach lattice together with a coarsely almost bounded pseudonorm $\rho'$. If an operator $T\colon (X,\rho)\to (Y,\norm{\cdot}_Y)$ is pseudo-$AM$-compact then $T$ is pseudo-semicompact with respect to $\rho$ and $\rho'$. 
\end{prop}
\begin{proof}
	Let $B\subseteq X$ be a pseudo-bounded with respect to $\rho$. Since $T\colon (X,\rho)\to (Y,\norm{\cdot}_Y)$ is pseudo-$AM$-compact, the set $T(B)$ is relatively compact in $Y$. Every relatively compact subset of $Y$ is also almost order bounded. Hence, $T(B)$ is almost order bounded in $Y$. Since the pseudonorm $\rho'$ is coarsely almost bounded, the set $T(B)$ is almost pseudo-bounded with respect to pseudonorm $\rho'$. This shows that the operator $T\colon X\to Y$ is pseudo-semicompact with respect to pseudonorms $\rho$ and $\rho'$.
\end{proof}

\begin{prop}
	Suppose that $(Y,\norm{\cdot}_Y)$ is a Banach lattice together with an order continuous pseudonorm $\rho'$. Let $(X,\rho)$ be a pseudonormed vector lattice. If an operator $T\colon (X,\rho)\to (Y,\norm{\cdot}_Y)$ is pseudo-$AM$-compact then $T\colon (X,\rho)\to (Y,\rho')$ is sequentially pseudonorm compact.
\end{prop}
\begin{proof}
	Let $x_n$ be a pseudo-bounded sequence in $X$. Since $T\colon (X,\rho)\to (Y,\norm{\cdot}_Y)$ is pseudo-$AM$-compact there exists a subsequence $x_{n_k}$ and $y\in Y$ such that $Tx_{n_k}\xrightarrow{\norm{\cdot}_Y}y$. As $Y$ is a Banach lattice, there is a subsequence $x_{n_{k_j}}$ such that $Tx_{n_{k_j}}\xrightarrow{o} y$ in $Y,$ see Example~\ref{AnExample} and~\cite[Theorem VII.2.1]{V1967}. As the pseudonorm $\rho'$ on $Y$ is order continuous, $Tx_{n_{k_j}}\xrightarrow{\rho'} y$. Hence, the operator $T\colon (X,\rho)\to (Y,\rho')$ is sequentially pseudonorm compact. 
\end{proof}

The following corollary results from Definition~\ref{DefnSemicompactVer}.
\begin{coro}
	Suppose that $T\colon X\to Y$ is a pseudo-semicompact operator and $S\colon Y\to Z$ is a pseudo-$AM$-compact operator where $X$ and $Y$ are pseudonormed vector lattices and $Z$ is a Banach lattice. Then the operator $ST$ is pseudonorm compact. 
\end{coro}
\begin{proof}
	Let $x_{\alpha}$ be a pseudo-bounded net in $X$. Then the net $Tx_{\alpha}$ is almost pseudo-bounded in $Y$. In particular, the net $Tx_{\alpha}$ is pseudo-bounded in $Y$. Since $S$ is pseudo-$AM$-compact, there exists a subnet $x_{\alpha_{\beta}}$ and $z\in Z$ such that $ST(x_{\alpha_{\beta}})\xrightarrow{\norm{\cdot}_Z} y$ in $Z$. This shows that the operator $ST\colon X\to Z$ is pseudonorm compact.
\end{proof}

In the rest of the present section, we focus on $M$-weakly compact and $L$-weakly compact operators. We recall that a continuous operator $T\colon X\to Y$ from a normed lattice $X$ into a normed space $Y$ is called \textit{$M$-weakly compact} if for every norm bounded disjoint sequence $x_n$ in $X$ we have $\lim_n\norm{Tx_n}=0,$ see~\cite[Section 5.3]{AB2} or~\cite[Section 3.6]{MN91}.

An operator $T\colon X\to Y$ between two pseudonormed vector lattices $X$ and $Y$ is said to be \textit{$M$-weakly compact with respect to pseudonorms} if for every pseudo-bounded disjoint sequence $x_n$ in $X$ one has $Tx_n\xrightarrow{} 0$ with respect to the pseudonorm on $Y$.

\begin{prop}
	Suppose that $T\colon (X,\rho)\to (Y,\rho')$ is an order continuous operator between two pseudonormed vector lattices $(X,\rho)$ and $(Y,\rho')$ where $\rho$ is a bounded pseudonorm and $\rho'$ is an order continuous pseudonorm. Then the operator $T$ is $M$-weakly compact with respect to pseudonorms $\rho$ and $\rho'$. 
\end{prop}
\begin{proof}
	Let $(x_n)\subseteq X$ be a disjoint sequence which is pseudo-bounded with respect to $\rho$. Since $\rho$ is a bounded pseudonorm, the sequence $x_n$ is order bounded in $X$. It follows that $x_n\xrightarrow{o}0$ in $X,$ see~\cite[Remark 10]{AEEM2018}. Since $T$ is order continuous, $Tx_n\xrightarrow{o}0$ in $Y$. As $\rho'$ is an order continuous pseudonorm on $Y,$ we have $Tx_n\xrightarrow{\rho'}0$. This shows that the operator $T\colon (X,\rho)\to (Y,\rho')$ is $M$-weakly compact with respect to $\rho$ and $\rho'$.
\end{proof}

\begin{prop}
	Suppose that $(X,\rho)$ and $(Y,\rho')$ are pseudonormed vector lattices where both $\rho$ and $\rho'$ are Riesz pseudonorms. Let $T\colon X\to Y$ be an $M$-weakly compact operator with respect to $\rho$ and $\rho'$. If $S\colon X\to Y$ is an operator satisfying $0\leq S\leq T$ then $S$ is $M$-weakly compact with respect to $\rho$ and $\rho'$.
\end{prop}
\begin{proof}
	Let $x_n$ be a disjoint sequence which is pseudo-bounded with respect to $\rho.$ Since $\rho$ is a Riesz pseudonorm, the disjoint sequence $|x_n|$ is also pseudo-bounded with respect to $\rho.$ As $T\colon X\to Y$ is $M$-weakly compact with respect to $\rho$ and $\rho'$, $T(|x_n|)\xrightarrow{} 0$ with respect to $\rho'$. It follows from $0\leq S\leq T$ that $0\leq S(|x_n|)\leq T(|x_n|)$ for all $n.$ Since $\rho'$ is a Riesz pseudonorm, we obtain $S(|x_n|)\xrightarrow{}0$ with respect to $\rho'$. It follows from $0\leq |Sx_n|\leq S(|x_n|)$ for all $n$ that $Sx_n\xrightarrow{} 0$ with respect to $\rho'$. Hence, the operator $S$ is $M$-weakly compact with respect to pseudonorms $\rho$ and $\rho'$. 
\end{proof}
A continuous operator $T\colon X\to Y$ from a normed space $X$ into a normed lattice $Y$ is called \textit{$L$-weakly compact} if for every disjoint sequence $y_n$ in the solid hull of $T(B_X),$ the image of the unit ball of $X$ under $T,$ one has $\lim_n \norm{y_n}=0,$ see~\cite[Section 5.3]{AB2} or~\cite[Section 3.6]{MN91}. 

An operator $T\colon (X,\rho)\to (Y,\rho')$ between two pseudonormed vector lattices $(X,\rho)$ and $(Y,\rho')$ is said to be \textit{$L$-weakly compact with respect to pseudonorms $\rho$ and $\rho'$} if for every disjoint sequence $y_n$ in the solid hull of $T(U_{\rho})$ in $Y$ one has $y_n\xrightarrow{} 0$ with respect to pseudonorm $\rho'$. 
\begin{prop} \label{PropLweaklyCompact}
	Suppose that $T\colon (X,\rho)\to (Y,\rho')$ is an order bounded operator between two pseudonormed vector lattices $(X,\rho)$ and $(Y,\rho')$ where $\rho$ is a bounded pseudonorm and $\rho'$ is an order continuous pseudonorm. Then the operator $T$ is $L$-weakly compact with respect to pseudonorms $\rho$ and $\rho$. 
\end{prop}
\begin{proof}
	As $\rho$ is an order bounded pseudonorm on $X,$ the set $U_{\rho}$ is order bounded in $X$. It follows that $T(U_{\rho})$ is order bounded in $Y$. Let $y_n$ be a disjoint sequence in $T(U_{\rho})$. As $y_n$ is order bounded, $y_n\xrightarrow{o}0$ in $Y,$ see~\cite[Remark 10]{AEEM2018}. As $\rho'$ is order continuous, $y_n\xrightarrow{} 0$ with respect to the pseudonorm $\rho'$. This shows that every disjoint sequence $y_n$ in $T(U_{\rho})$ satisfies $y_n\xrightarrow{}0$ with respect to pseudonorm $\rho'$. Hence, the operator $T\colon (X,\rho)\to (Y,\rho')$ is $L$-weakly compact with respect to pseudonorms $\rho$ and $\rho'$.
\end{proof}
In view of Proposition~\ref{PropLweaklyCompact}, we restate some an implication given in~\cite[Section 3.6]{MN91}. Let $X$ and $Y$ be Banach lattices. If $Y$ is order continuous then every semicompact operator $T\colon X\to Y$ is $L$-weakly compact.

It is known that order continuous operators are order bounded, see~\cite[Lemma 1.54]{AB2}. However, there are sequentially order continuous operators which are not order bounded. Several examples can be found in~\cite{AB2}.
\begin{thm}
	Suppose that $T\colon (X,\rho)\to (Y,\rho')$ is an order bounded and sequentially order continuous operators between two pseudonormed vector lattices $(X,\rho)$ and $(Y,\rho')$ where $\rho$ is bounded and $\rho'$ is order continuous. Then the operator $T$ is both $M$-weakly compact and $L$-weakly compact with respect to pseudonorms $\rho$ and $\rho'$. 
\end{thm}
\begin{proof}
	Let $(x_n)\subseteq X$ be a disjoint sequence which is pseudo-bounded with respect to $\rho$. Since $\rho$ is a bounded pseudonorm, the sequence $x_n$ is order bounded in $X$. It follows that $x_n\xrightarrow{o}0$ in $X$. Since $T$ is sequentially order continuous, $Tx_n\xrightarrow{o}0$ in $Y$. As $\rho'$ is an order continuous pseudonorm on $Y,$ we have $Tx_n\xrightarrow{}0$ with respect to $\rho'$. This shows that the operator $T\colon (X,\rho)\to (Y,\rho')$ is $M$-weakly compact with respect to pseudonorms.
	
	Let us show that $T$ is $L$-weakly compact with respect to pseudonorms. Since the set $U_{\rho}$ is order bounded in $X,$ the set $T(U_{\rho})$ is order bounded in $Y$. Let $y_n$ be a disjoint sequence in $T(U_{\rho})$. As $y_n$ is order bounded, $y_n\xrightarrow{o}0$ in $Y$. As $\rho'$ is order continuous, $y_n\xrightarrow{} 0$ with respect to the pseudonorm $\rho'$. This shows that every disjoint sequence $y_n$ in $T(U_{\rho})$ satisfies $y_n\xrightarrow{}0$ with respect to pseudonorm $\rho'$. Hence, the operator $T\colon (X,\rho)\to (Y,\rho')$ is $L$-weakly compact with respect to pseudonorms $\rho$ and $\rho'$.
\end{proof}
\begin{prop}
	Suppose that $(X,\rho)$ and $(Y,\rho')$ are pseudonormed vector lattices where both $\rho$ and $\rho'$ are Riesz pseudonorms. Let $T\colon X\to Y$ be an $L$-weakly compact operator with respect to $\rho$ and $\rho'$. If $S\colon X\to Y$ is an operator satisfying $0\leq S\leq T$ then $S$ is $L$-weakly compact with respect to pseudonorms.
\end{prop}
\begin{proof}
	As $\rho$ is a Riesz pseudonorm, if $x\in U_{\rho}$ then $|x|\in U_{\rho}.$ This implies that the solid hull of $S(U_{\rho})$ is a subset of the solid hull of $S(|U_{\rho}|)\subseteq T(U_{\rho})$ in $Y$. Let $y_n$ be a disjoint sequence in $S(U_{\rho})$. As $T$ is $L$-weakly compact, $y_n\xrightarrow{} 0$ in the pseudonorm $\rho'$. This shows that the operator $S$ is $L$-weakly compact. 
\end{proof}
\begin{prop}
	Suppose that $T\colon (X,\rho)\to (Y,\rho')$ is an $L$-weakly compact lattice homomorphism between seminormed vector lattices $(X,\rho)$ and $(Y,\rho')$. Then $T$ is $M$-weakly compact with respect to seminorms $\rho$ and $\rho'$. 
\end{prop}
\begin{proof}
	Let $(x_n)\subseteq X$ be a disjoint sequence which is pseudo-bounded with respect to $\rho$. Since $\rho$ is a Riesz seminorm, there exists an $\epsilon>0$ such that the sequence $x_n$ is contained in $\epsilon U_{\rho}.$ Since $T$ is a lattice homomorphism, the sequence $Tx_n$ is disjoint and it is contained in $\epsilon T(U_{\rho}).$ As $T$ is $L$-weakly compact with respect to $\rho$ and $\rho'$, we have $Tx_n\xrightarrow{} 0$ with respect to seminorm $\rho'$. This shows that the homomorphism $T$ is $M$-weakly compact with respect to seminorms $\rho$ and $\rho'$. 
\end{proof}

\section{Montel Operators and the Lebesgue Property}\label{Fatou}
A linear topology $\tau$ on the vector lattice $X$ is said to be locally solid if $\tau$ has a base at zero consisting of solid sets. By~\cite[Theorem 2.28]{AB1}, a linear topology $\tau$ is locally solid if and only if there exists a family $\{\rho_i \}_{i\in I}$ consisting of continuous Riesz pseudonorms generating $\tau$. A systematic study of operators on topological vector spaces having compactness properties can be found in~\cite{K1979}.

\begin{definition}\label{DefinitionMontel}
		An operator $T\colon X\to Y$ between two locally solid vector lattices $X$ and $Y$ is said to be \textit{Montel} if for any topologically bounded net $x_{\alpha}$ in $X,$ there is a subnet $x_{\alpha_\beta}$ and $y\in Y$ such that the net $Tx_{\alpha_\beta}\to y$ in $Y$. The operator $T\colon X\to Y$ is called compact if there exists some zero neighborhood $U$ in $X$ such that for every net $x_{\alpha}$ in $U$ there exists a subnet $x_{\alpha_{\beta}}$ and $y\in Y$ satisfying $Tx_{\alpha_{\beta}}\xrightarrow{} y.$
\end{definition}

The classes of Montel and compact operators on locally solid spaces are quite different. It follows from Definition~\ref{DefinitionMontel} that every compact operator $T\colon X\to Y$ is Montel. However, both of these classes generalize the classic compact operators on Banach lattices. Indeed, if $X$ and $Y$ are Banach lattices, an operator $T\colon X\to Y$ is Montel if and only if it is compact.

\begin{exam}
	It is instructive to show that vector lattices and their order duals can be used to produce examples of compact and Montel operators. Let $X$ be a vector lattice, $A\subseteq X^{\sim}$ be nonempty where $X^{\sim}$ stands for the order dual of $X$. Consider the locally convex-solid vector lattice $(X,|\sigma|(X,A))$ where $|\sigma|(X,A)$ denotes the absolute weak topology on $X$ generated by $A,$ see~\cite[Definition 2.32]{AB1}. Let $B$ be a finite nonempty subset of the order ideal generated by $A$ in $X^{\sim}.$ The operator $T\colon (X,|\sigma|(X,A))\to (X,|\sigma|(X,A))$ defined by $T(x)=\sum_{f\in B}|f|(|x|)x$ is a Montel operator. 
\end{exam}

\begin{exam}
	Let $(X,\norm{\cdot}_X)$ and $(Y,\norm{\cdot}_Y)$ be normed vector lattices, and, $u\in X^+$. It follows from~\cite[Theorem 2.1]{EV2017} or Lemma~\ref{UnboundedRiesz} that the pseudonorm $\rho_u=\norm{|x|\wedge u}_X$ generates a locally solid topology on $X$. If an operator $T\colon (X,\rho_u)\to (Y,\norm{\cdot}_Y)$ is compact in the sense of Definition~\ref{DefinitionMontel} then $T\colon (X,\norm{\cdot}_X)\to (Y,\norm{\cdot}_Y)$ is compact. Indeed, let $x_{\alpha}$ be a bounded net in $(X,\norm{\cdot}_X).$ Evidently, the net $x_{\alpha}$ belongs to a locally solid neighborhood in $(X,\rho_u).$ Thus, the operator $T\colon (X,\norm{\cdot}_X)\to (Y,\norm{\cdot}_Y)$ is compact.
\end{exam}
 
\begin{prop}\label{PFatou1}
	Suppose that $X$ and $Y$ are locally solid vector lattices. Denote by $\{\rho_i \}_{i\in I}$ and $\{\rho'_j \}_{j\in J}$ families of Riesz pseudonorms generating the locally solid topologies on $X$ and $Y,$ respectively. The following statements are equivalent. 
	\begin{itemize}
		\item[\em i.] {The operator $T\colon X\to Y$ is Montel. } 
		\item[\em ii.] {For every net $(x_{\alpha})_{\alpha}\subseteq X$ satisfying the property that for every $i\in I$ there exists $\lambda_i$ such that $(\lambda_ix_\alpha)_{\alpha}\subseteq U_{\rho_i}$ there exists a subnet $x_{\alpha_{\beta}}$ and $y\in Y$ such that $Tx_{\alpha_{\beta}}\xrightarrow[\beta]{\rho'_j} y$ for every $j\in J$.}
	\end{itemize}
\end{prop}
\begin{proof}
	$(i)\Rightarrow (ii).$ Suppose that $x_{\alpha}$ satisfies the property that for every $i\in I$ there exists some $\lambda_i$ such that $(\lambda_i x_{\alpha})_{\alpha}\subseteq U_{\rho_i}.$ Let $U$ be a solid $\tau$-neighborhood of zero. Since $\{\rho_i \}_{i\in I}$ generates $\tau,$ there is some $i\in I$ such that $U_{\rho_i}\subseteq U.$ Hence, $\{\lambda_i x_{\alpha}\}_{\alpha}\in U.$ This shows that $x_{\alpha}$ is a topologically bounded net in $X$. As $T\colon X\to Y$ is a Montel operator, there exists a subnet $x_{\alpha_{\beta}}$ and $y\in Y$ such that $Tx_{\alpha_{\beta}}\xrightarrow[\beta]{\tau} y.$ This implies, by~\cite[Theorem 2.28]{AB1}, that $Tx_{\alpha_{\beta}}\xrightarrow[\beta]{\rho'_j} y$ for each $j\in J.$ 
	
	$(ii)\Rightarrow (i).$ If $Tx_{\alpha_{\beta}}\xrightarrow[\beta]{\rho'_j} y$ for every $j\in J$ then $Tx_{\alpha_{\beta}}\xrightarrow{} y$ with respect to the locally solid topology on $Y,$ by~\cite[Theorem 2.28]{AB1}. This is equivalent to saying that $T\colon X\to Y$ is a Montel operator.
\end{proof}

\begin{rem}\label{RemContinuity}
	Let $\rho$ and $\rho'$ be Riesz pseudonorms on vector lattices $X$ and $Y,$ respectively. Both $\rho$ and $\rho'$ generate locally solid topologies on the underlying spaces. For an arbitrary operator $T\colon (X,\rho)\to (Y,\rho')$ there are two nonequivalent notions of boundedness. In details, an operator $T\colon (X,\rho)\to (Y,\rho')$ is said to be neighborhood bounded if there is a zero neighborhood $U\subseteq X$ such that $T(U)$ is bounded in the pseudonormed vector lattice $(Y,\rho')$. It is called boundedly bounded if the images of bounded subsets of $(X,\rho)$ are bounded in $(Y,\rho')$. We remark that various results related to neighborhood bounded and boundedly bounded operators which are simultaneously compact or Dunford-Pettis operators with respect to unbounded convergences and locally solid topologies can be found in~\cite{EGZ2018, EGZ2019}.
\end{rem}

\begin{prop}
	Suppose that $X$ is a locally solid vector lattice, and that, $Y$ is a Dedekind complete normed lattice. If an operator $T\colon X\to Y$ is Montel then $T$ is $AM$-compact, i.e., $T[-x,x]$ is relatively compact for every $x\in X^+$. Conversely, if there exists a bounded pseudonorm $\rho$ on $X$ such that $U_{\rho}$ is a neighborhood of zero in $X$ then every $AM$-compact operator is Montel.
\end{prop}
\begin{proof}
	 Let $x_{\alpha}$ be an order bounded net in $X$. Since order bounded nets are topologically bounded in locally solid vector lattices, and, $T\colon X\to Y$ is a Montel operator; there exists a subnet $x_{\alpha_{\beta}}$ and $y\in Y$ such that $Tx_{\alpha_{\beta}}\xrightarrow{}y$ in $Y$. Hence, if $[-x,x]$ is an order interval in $X$ then the set $T[-x,x]$ is relatively compact in $Y$. This shows that the operator $T\colon X\to Y$ is $AM$-compact.
	
 	Conversely, let $x_{\alpha}$ be a topologically bounded net in $X$. It is given that the set $U_{\rho}$ is a zero neighborhood in $X$. By the proof of Proposition~\ref{PFatou1}, there exists some $\lambda$ such that $\lambda x_{\alpha}\in U_{\rho}$ for all $\alpha$. Since $\rho$ is a bounded pseudonorm, $U_{\rho},$ in particular the net $x_{\alpha},$ is order bounded in $X$. As $T\colon X\to Y$ is $AM$-compact, there exists a subnet $x_{\alpha_{\beta}}$ and $y\in Y$ such that $Tx_{\alpha_{\beta}}\xrightarrow{}y$ in $Y$. Hence, the operator $T\colon X\to Y$ is Montel.
\end{proof}

In the following result, we compare locally solid and unbounded locally solid topologies on both the domain and the range of an operator. We recall from Section~\ref{PseudoBasics}, also see~\cite{EV2017}, that $\rho_u(x)=\rho(|x|\wedge u)$ for $x\in X$ and $u\in X^+$. 
\begin{prop}
	Suppose that $X$ and $Y$ are locally solid vector lattices whose topologies are generated by $\{\rho_i \}_{i\in I}$ and $\{\rho'_j \}_{j\in J},$ respectively. Suppose further that there exist $u_0\in X^+$ and $i_0\in I$ such that $U_{(\rho_{i_0})_{u_0}}$ is a neighborhood of zero in $X$ where $(\rho_{i_0})_{u_0}$ is a bounded pseudonorm. If an operator $T\colon X\to Y$ is Montel then the operator $T$ is Montel with respect to locally solid topologies generated by $\{(\rho_i)_u \}_{i\in I,u\in X^+}$ and $\{(\rho'_j)_v \}_{j\in J,v\in Y^+},$ respectively.
\end{prop}
\begin{proof}
	Let $x_{\alpha}$ be a net which is topologically bounded with respect to the locally solid topology generated by $\{(\rho_i)_u \}_{i\in I,u\in X^+}$ on $X$. Since $U_{(\rho_{i_0})_{u_0}}$ is a neighborhood of zero in $X,$ there exists a $\lambda\geq 0$ such that $(\lambda x_{\alpha})_{\alpha}\subseteq U_{(\rho_{i_0})_{u_0}}.$ In particular, the net $\lambda x_{\alpha}$ is pseudo-bounded with respect to $(\rho_{i_0})_{u_0}.$ Since $(\rho_{i_0})_{u_0}$ is a bounded pseudonorm, the net $\lambda x_{\alpha}$ is order bounded in $X$. Hence, the net $x_{\alpha}$ is order bounded in $X$. This implies that the net $x_{\alpha}$ is topologically bounded with respect to the locally solid topology generated by $\{\rho_i \}_{i\in I}$. Since $T\colon X\to Y$ is Montel, there exists a subnet $x_{\alpha_{\beta}}$ and $y\in Y$ such that $Tx_{\alpha_{\beta}}\xrightarrow{} y$ in $Y$ with respect to the locally solid topology generated by the pseudonorms $\{\rho'_j \}_{j\in J}$. Hence, $Tx_{\alpha_{\beta}}\xrightarrow{} y$ in $Y$ with respect to the locally solid topology generated by $\{(\rho'_j)_v \}_{j\in J,v\in Y^+}$.
\end{proof}
We recall from~\cite[Theorem 4.6]{EV2017} that a locally solid topology on a vector lattice $X$ has the Lebesgue property if and only if the corresponding unbounded locally solid topology has the Lebesgue property. 
\begin{prop}\label{MontelOrder}
	Suppose that $X$ is a locally solid vector lattice, and that, $Y$ is a locally solid vector lattice having the Lebesgue property. If an operator $T\colon X\to Y$ is Montel then $T$ is order compact. 
\end{prop}
\begin{proof}
	Let $x_{\alpha}$ be an order bounded net in $X.$ Since the linear topology on $X$ is locally solid, the net $x_{\alpha}$ is topologically bounded in $X$. As the operator $T\colon X\to Y$ is Montel, there exists a subnet $x_{\alpha_{\beta}}$ and $y\in Y$ such that $Tx_{\alpha_{\beta}}\xrightarrow{} y$ with respect to the locally solid topology on $Y$. Since $Y$ has the Lebesgue property, we have $Tx_{\alpha_{\beta}}\xrightarrow{o} y$. Hence, the operator $T\colon X\to Y$ is order compact. 
\end{proof}
\begin{rem}
	In view of Proposition~\ref{MontelOrder}, let $T\colon X\to Y$ be a Montel operator where $X$ is a locally solid vector lattice and $Y$ is a Hausdorff locally solid vector lattice. In this case, for every zero neighborhood $U$ of $X$ there exists some $y\in Y$ such that $y=\sup_{x\in U} y\wedge Tx.$ When $Y$ has that Lebesgue property, $y\in Y$ can be taken from the order closure of $T(U).$
\end{rem}
\begin{prop}\label{MontelOrder2}
	Suppose that $X$ and $Y$ are locally solid vector lattices, and that, there exists a continuous pseudonorm $\rho'$ on $Y$ having the subnet property. If an operator $T\colon X\to Y$ is Montel then $T$ is order compact.
\end{prop}
\begin{proof}
	Let $x_{\alpha}$ be an order bounded net in $X.$ As the operator $T\colon X\to Y$ is Montel, there exists a subnet $x_{\alpha_{\beta}}$ and $y\in Y$ such that $Tx_{\alpha_{\beta}}\xrightarrow{} y$ with respect to the locally solid topology on $Y$. Since the pseudonorm $\rho'$ on $Y$ is continuous, $Tx_{\alpha_{\beta}}\xrightarrow{} y$ with respect to the pseudonorm $\rho'$. As $\rho'$ has the subnet property, there exists $x_{\alpha_{\beta_{\gamma}}}$ such that $Tx_{\alpha_{\beta_{\gamma}}}\xrightarrow{o} y$. Hence, the operator $T\colon X\to Y$ is order compact. 
\end{proof}
\begin{rem}
	In Proposition~\ref{MontelOrder} and Proposition~\ref{MontelOrder2}, we use the structure on the range space $Y$ of a Montel operator $T\colon X\to Y$ to conclude that it is an order compact operator. We record from~\cite[24L.(a)]{F1974} that any continuous Riesz pseudonorm on locally solid space $Y$ having the Lebesgue property is a pseudonorm having the Fatou property, see Definition~\ref{Def002}.
\end{rem}

\begin{prop}
	Suppose that $X$ and $Y$ are locally solid vector lattices. Suppose that there exists a bounded pseudonorm $\rho$ on $X$ such that $U_{\rho}$ is a neighborhood of zero in $X$. Suppose further that $Y$ has the Lebesgue property. If an operator $T\colon X\to Y$ is order compact then $T$ is Montel. 
\end{prop}
\begin{proof}
	Let $x_{\alpha}$ be a topologically bounded net in $X$. Since $\rho$ is a bounded pseudonorm on $X$, the set $U_{\rho}$ is order bounded. There exists a $\lambda\geq 0$ such that $\lambda x_{\alpha}\in U_{\rho}$ for every $\alpha$. Hence, the net $\lambda x_{\alpha},$ and in particular the net $x_{\alpha},$ are order bounded in $X$. As $T\colon X\to Y$ is order compact, there exists a subnet $x_{\alpha_{\beta}}$ and $y\in Y$ such that $Tx_{\alpha_{\beta}}\xrightarrow{o} y.$ Since $Y$ has a Lebesgue topology, $Tx_{\alpha_{\beta}}\xrightarrow{} y$ with respect to the locally solid topology on $Y$. This shows that the operator $T$ is Montel.
\end{proof}

\begin{prop}\label{bbOperator}
	Suppose that $X$ and $Y$ are locally solid vector lattices. Let $\rho'$ be a continuous bounded pseudonorm on $Y$. If an operator $T\colon X\to Y$ is boundedly bounded then $T$ maps topologically bounded sets into almost order bounded sets.
\end{prop}
\begin{proof}
	Let $B\subseteq X$ be a topologically bounded subset. Since $T(B)$ is topologically bounded in $Y,$ there exists some $\lambda>0$ such that $ T(B)\subseteq \lambda U_{\rho'}.$ As $\rho'$ is a bounded pseudonorm on $Y$, the set $T(B)$ is order bounded in $Y$. Hence, $T(B)$ is almost order bounded in $Y$.
\end{proof}

\begin{prop}
	Suppose that $X$ is a barrelled locally convex-solid vector lattice, and that, $Y$ is a locally solid vector lattice with the Lebesgue property. Let $\rho_f(x)=|f(x)|$ and $\rho_g(x)=|g(x)|$ with $f\in X',$ the topological dual of $X,$ and $g\in Y'$ be seminorms on $X$ and $Y,$ respectively. If $\rho_g$ has the subnet property then every pseudonorm compact $T\colon (X,\rho_f)\to (Y,\rho_g)$ is a Montel operator.
\end{prop}
\begin{proof}
	Let $x_{\alpha}$ be a topologically bounded net in $X$. Since $f\in X'$ and $X$ is a barrelled space, the set $U_{\rho_f}=\{x\in X\colon \rho_f(x)\leq 1 \}$ is a zero neighborhood in $X,$ see the proof of~\cite[Theorem 2.29]{AB2}. Hence, the net $x_{\alpha}$ is bounded with respect to seminorm $\rho_f$. As $T$ is pseudonorm compact, there exists a subnet $x_{\alpha_{\beta}}$ and $y\in Y$ such that $Tx_{\alpha_{\beta}}\xrightarrow{\rho_g} y$ in $Y$. As $\rho_g$ has the subnet property, there exists a subnet $x_{\alpha_{\beta}}$ such that $Tx_{\alpha_{\beta}}\xrightarrow{o} y$ in $Y$. Since $Y$ has Lebesgue topology, we have $Tx_{\alpha_{\beta}}\xrightarrow{} y$ with respect to locally solid topology of $Y$. This shows that $T$ is a Montel operator.
\end{proof}

{\bf Acknowledgments} 
This work was partially supported by TUBITAK Project 118F204. We would like to thank the referee for her/his valuable suggestions.

\end{document}